\documentclass[review,12pt]{elsarticle}
\usepackage{amsthm,amssymb,amsmath}
\usepackage{mathrsfs}
\usepackage{graphicx,subfigure,color}
\usepackage{cases}
\usepackage{exscale}
\usepackage{relsize}
\usepackage{tabularx}
\usepackage{lineno}

\definecolor{forestgreen}{cmyk}{0.91,0,0.88,0.12}
\definecolor{darkorange}{rgb}{1.00,0.55,0.00}
\definecolor{violet}{rgb}{0.5,0,0.8}

\newtheorem{theorem}{Theorem}[section]
\newtheorem{lemma}{Lemma}[section]

\newtheorem{comment}{Comment}[section]

\numberwithin{equation}{section}

\begin{document}
\begin{frontmatter}
\title{Numerical analysis of a semi-implicit Euler scheme for the Keller-Segel model}
\tnotetext[]{X. Huang's work is supported by China Scholarship Council grant 202306310119; O. Goubet's work is supported by Labex CEMPI (ANR-11-LABX-0007-01); J. Shen's work is partially supported by NSFC grant 12371409.}

\author[label1]{Xueling Huang}
\ead{hxlmath@163.com}
\author[label2]{Olivier Goubet\corref{cor1}}
\ead{olivier.goubet@univ-lille.fr}\cortext[cor1]{Corresponding author.}
\author[label3,label1]{Jie Shen}
\ead{jshen@eitech.edu.cn}

\address[label1]{School of Mathematical Sciences and Fujian Provincial Key Laboratory on Mathematical Modeling and High Performance Scientific Computing, Xiamen University, Xiamen, Fujian, 361005, China.}
\address[label2]{Laboratoire Paul Painlev$\acute{e}$ CNRS UMR 8524, et $\acute{e}$quipe projet INRIA PARADYSE, Universit$\acute{e}$ de Lille, Lille, 59000, France.}
\address[label3]{School of Mathematical Science, Eastern Institute of Technology, Ningbo, Zhejiang 315200, China.}

\begin{abstract}
We study the properties of a semi-implicit Euler scheme that is widely used in time discretization of Keller-Segel equations both in the parabolic-elliptic form and the parabolic-parabolic form.
We prove that this linear, decoupled, first-order scheme preserves unconditionally the important properties of Keller-Segel equations  at the semi-discrete level, including the mass conservation and positivity preserving of the cell density, and the energy dissipation.
We also establish optimal  error estimates in $L^p$-norm $(1<p<\infty)$.
\end{abstract}
\end{frontmatter}

\section{Introduction}
Chemotaxis refers to the directional migration of cells along the concentration gradient of a certain chemoattractant.
In the 1970s, Keller \cite{KS1970} and Segel \cite{KS1971} established the first mathematical model describing chemotaxis, which is also considered as a further development of the work of Patlak \cite{PKS1953}.
In this paper we focus on the general dimensionless Keller-Segel (KS) model \cite{Nagai1997}
\begin{align}
&\frac{\partial \rho}{\partial t}
= \Delta \rho- \chi\nabla \cdot \left(\rho \nabla c \right), \ \ && \boldsymbol{x}\in \Omega, \ t>0, \label{2.1} \\
&\tau \frac{\partial c}{\partial t}= \Delta c - \alpha c + \gamma\rho, \ \ &&\boldsymbol{x}\in \Omega, \ t>0, \label{2.2}
\end{align}
subjected to suitable initial condition
\begin{equation}\label{2.3}
\rho|_{t=0} = \rho_{0}\geq 0, \ \ \tau c|_{t=0} = \tau c_{0}\geq 0, \ \ \text{and} \ \ \rho_{0}, c_0 \not\equiv 0
\end{equation}
in a bounded domain $\Omega\subset\mathbf{R}^{2}$ with smooth boundary $\partial\Omega$.
We consider the following boundary conditions:
\begin{itemize}
  \item periodic boundary conditions for the cell density $\rho=\rho(\boldsymbol{x},t)$ and the concentration of chemoattractants $c=c(\boldsymbol{x},t)$; or
  \item Neumann boundary conditions for $\rho$ and $c$, that is
      \begin{equation}\label{1.4}
      \frac{\partial \rho}{\partial \boldsymbol{n}}=  \frac{\partial c}{\partial \boldsymbol{n}}=0,\ \ \boldsymbol{x}\in \partial\Omega, \ t>0
      \end{equation}
      where $\boldsymbol{n}$ is the outward unit-normal to the boundary $\partial\Omega$.
\end{itemize}
In the above, parameters $\chi>0$ is the sensitivity of cells to the chemoattractant,
$\alpha, \gamma>0$ represents the consumption and production rate of chemoattractant, respectively.
The model is a parabolic-parabolic system when $\tau > 0$, and
a parabolic-elliptic system when $\tau = 0$, which means that chemoattractant diffusion is much faster than cell diffusion.
Note that if $\tau = 0$, only the initial data $\rho_0$ is needed in \eqref{2.3}.

There have been many mathematical studies on the KS model above.
Numerous researches focused on global existence and boundedness of solutions along time (see, e.g. \cite{TW2015,Horstmann2003,Hillen2009,Nagai2001} and references therein).
In two-dimensional (2D) domains, the global existence depends on a threshold: when the initial mass lies below the threshold solutions exist globally and are $L^{\infty}$ bounded, while above the threshold solutions blow up and are unbounded in $L^{\infty}$-norm.
Espejo et al. \cite{EEE2012,EEE2009,EEE2013} analysed the blow-up conditions of the KS system in a 2D ball with a radial symmetric setting, and in in the whole space $\mathbf{R}^{2}$ without the radially symmetric setting.
For the existence results of solutions to KS equations in three or higher dimensional domains, we refer to \cite{A2021,TW2015,C2004,Horstmann2003} and references therein.
Since the existence of solutions in higher space dimensions is different from that in 2D, we only focus on the 2D problem in this paper.

The KS equations  enjoy some important properties, such as mass conservation and positivity preserving for the cell density, and
the energy dissipation (see, e.g. \cite{Nagai1995,G1998,Nonnegative,Nagai1997,Saito2005} and references therein).
A good numerical method should preserve these properties at the discrete level as much as possible, and  its implementation should be simple and amenable to  rigorous numerical analysis.

Over the years, various  numerical schemes  have been developed for the KS equations (see, for instance, \cite{Saito2012,GS2021,Chat2022,Acosta2023} and the references therein).
Most positive-preserving  schemes  depend on a particular spatial discretization and may lead to strict CFL restrictions on the time step.
For examples, finite difference and finite element methods are often combined with upwind techniques \cite{Liu2018,Saito2012,Saito2005}, flux correction methods \cite{Chat2022,HXL2020,Strehl2010}, total variation diminishing methods \cite{Epshteyn2009}, etc.
The shortcoming is that the energy dissipation is often lost as a result.
Zhou and Saito \cite{Zhou2021,Zhou2017} proposed special finite volume schemes that satisfy both positivity preserving and energy dissipation, and derived error estimates, while their calculations are usually not simple and flexible enough. More recently,
 fully discrete schemes which enjoy all the essential properties have been constructed by finite element discretization \cite{GS2021}, finite difference method \cite{Lu2023}, and discontinuous Galerkin method \cite{Acosta2023}. However, these papers lack the error analysis for their numerical schemes.

Another interesting scheme  is  proposed in \cite{SheX20} in which the authors  constructed a semi-discrete (in time)  nonlinear scheme based on the structure of the Wasserstein gradient flow, and showed that it enjoys unconditionally  all essentially properties of the KS equations, and the nonlinear equation is uniquely solvable as at each time step. A rigorous error analysis is carried out in \cite{CLS22} for a fully discrete scheme based on this time discretization. A distinct advantage of this scheme is that its properties can be easily carried over to many spatial discretizations, but a   main drawback of this scheme is that one needs to solve a coupled nonlinear system at each time step.

Although the general KS model \eqref{2.1}-\eqref{1.4} is nonlinear and involves two variables $\rho$ and $c$ that are coupled together, for the sake of simplicity and efficiency, one still favor a proper time-discrete scheme which is linear and decoupled.
For example,  the following semi-discrete scheme is often applied in actual computations \cite{GS2021,Saito2005}:
\begin{align}
&\frac{\rho^{n+1}-\rho^{n}}{\delta t} =
 \Delta \rho^{n+1} - \chi\nabla \cdot \left(\rho^{n+1} \nabla c^{n} \right), \label{S1}\\
&\tau\frac{c^{n+1}-c^{n}}{\delta t} = \Delta c^{n+1} - \alpha c^{n+1} + \gamma \rho^{n+1}, \label{S2}
\end{align}
with the periodic or Neumann boundary conditions,
where $\delta t$ is the time step and $\rho^{n}$ and $c^{n}$ are approximations of $\rho$ and $c$ at the $n$th time step.
At each time step, as long as ($\rho^{n}$,  $c^{n}$), $j=0,1,2,...,n$ are given, it is efficient to solve $\rho^{n+1}$ from \eqref{S1} and $c^{n+1}$ from \eqref{S2} since they are both linear elliptic equations.
As time derivatives are approximated by the stable backward Euler method, linear terms are implicitly approximated and nonlinear terms are semi-implicitly approximated.
There is sufficient numerical evidence that, with a proper spatial discretization,  this semi-implicit Euler scheme appear to be able to  preserve essential properties, such as mass conservation, energy dissipative and positivity preserving,  of the KS equations \cite{GS2021,Saito2012,Saito2005}.
However, it has not been rigorously proven that this scheme, at the semi-discrete lever, can preserve all the above properties.

First of all, it is easy to obtain the discrete mass conservation by integrating \eqref{S1}-\eqref{S2} over the domain $\Omega$ with periodic or Neumann boundary conditions for the cell density.

For the energy dissipation, it was firstly shown in \cite{Saito2005}  that   the semi-implicit Euler scheme \eqref{S1}-\eqref{S2} is energy dissipative. It is also worthy to note that
 Liu et al. \cite{Liu2018} reformulated the density equation \eqref{2.1} with $\chi=1$ as $\frac{\partial \rho}{\partial t} = \nabla \cdot \left( e^c \nabla \left(\frac{\rho}{e^c} \right) \right)$, and pointed out the equivalence between the semi-discrete scheme
\begin{align}\label{1.7}
	\frac{\rho^{n+1}-\rho^{n}}{\delta t} =
	\nabla \cdot \left( e^{c^{n}} \nabla \left(\frac{\rho^{n+1}}{e^{c^{n}}} \right) \right)
\end{align}
and the semi-implicit Euler equation \eqref{S1}. But they only carried out stability analysis  for a related coupled nonlinear scheme where $c^n$ in  \eqref{S1} is replaced by $c^{n+1}$.

Regarding the positivity of numerical solution $\rho^n$,
there seems to be a general consensus among researchers that this scheme appears to satisfy  the maximal principle, but to the best of our knowledge, there is no rigorous proof of this fact.

In addition, to the best of our knowledge, there is no error analysis for the semi-implicit Euler scheme \eqref{S1}-\eqref{S2}. Moreover, as pointed out  in  \cite{Liu2018,Saito2005}, naive space discretizations of this semi-discrete scheme can easily destroy the positivity of the solution and trigger instability.

\begin{comment}
A natural idea is to perform spatial discretization based on the semi-discrete scheme \eqref{S1}-\eqref{S2} to construct fully discrete  schemes which can preserve the properties (mass conservation, positivity preserving and energy dissipation) at the fully discrete level.
Unfortunately, the authors in  \cite{Liu2018,Saito2005}  pointed out that naive space discretizations of this semi-discrete scheme can easily destroy the positivity of the solution and trigger instability.
Therefore, additional correction techniques are often added when performing spatial discretization. For examples, finite difference and finite element methods are often combined with upwind techniques \cite{Liu2018,Saito2012,Saito2005}, flux correction methods \cite{Chat2022,HXL2020,Strehl2010}, total variation diminishing methods \cite{Epshteyn2009}, etc.
The shortcoming is that the energy stability is often lost as a result.
Zhou and Saito \cite{Zhou2021,Zhou2017} proposed special finite volume schemes that satisfies both positivity preserving and energy dissipation, and derived error estimates, while their calculation is usually not simple and flexible enough.
Very recently, fully discrete schemes which enjoy all the essential properties have been constructed by finite element discretization \cite{GS2021}, finite difference method \cite{Lu2023}, and discontinuous Galerkin method \cite{Acosta2023}.
\end{comment}

To sum up the above, a comprehensive and rigorous numerical analysis of the semi-implicit Euler scheme \eqref{S1}-\eqref{S2} will not only enable us to fully understand this  semi-discrete scheme, but  also provide a basis for further error analysis of fully discrete schemes combined with this semi-discretization. The main purpose of this paper is to conduct a  rigorous numerical analysis for the semi-implicit Euler scheme \eqref{S1}-\eqref{S2}. The main contributions of this paper include
\begin{itemize}
	\item providing a rigorous proof that the scheme \eqref{S1}-\eqref{S2} satisfies unconditionally mass conservation, energy dissipative and positivity preserving;
	\item obtaining $L^p$-bounds ($1<p<\infty$) for the numerical solution of  \eqref{S1}-\eqref{S2}; and
	\item deriving optimal error estimates in $L^p$-norm ($1<p<\infty$).
\end{itemize}

The rest of paper is organized as follows.
In Section 2, we introduce and prove basic properties of the KS model \eqref{2.1}-\eqref{1.4}, including mass conservation, positivity preserving and energy dissipation.
In Section 3, a series of bounds of the cell density $\rho$ are preparatory.
In Section 4, we analysis the semi-implicit Euler scheme \eqref{S1}-\eqref{S2} keeping mass conservation, positivity preserving of the semi-discrete solution $\rho^{n}$ and the semi-discrete energy dissipation.
In Section 6, error estimates in $L^p$-norm ($1<p<\infty$) between the exact solutions ($\rho$, $c$) and semi-discrete numerical solutions ($\rho^{n+1}$, $c^{n+1}$) are derived by $L^p$-bounds ($1<p<\infty$) of $\rho^{n+1}$ for all $n$ obtained in Section 5.
We end the paper with some concluding remarks in Section 7.

We now describe some notations.Throughout the paper $W^{m,p}=W^{m,p}(\Omega)$ denotes
a Sobolev space and $H^m=W^{m,2}$ denotes a Hilbert space with inner product $(\cdot, \cdot)_{H^m}$ and norm $\|\cdot\|_{H^m}$.
Besides, $(\cdot, \cdot)$ is the $L^2$ inner product and $\|\cdot\|_{L^p}$ is the $L^p$-norm on the domain $\Omega$.
We denote by $\mathcal{L}(E,F)$ the space of bounded linear operators from the normed space $E$ into $F$.
Moreover, we shall use bold faced letters to denote vectors and vector spaces, and use
$C$ to denote a generic positive constant independent of discretization parameters.
The constant $C$ may change value from one line to one another without notice.

\section{Properties of the KS model}
There have been a number of reliable studies on the well-posedness of the initial-boundary value problem \eqref{2.1}-\eqref{1.4} (see, e.g. \cite{C2004,Nagai1995,Nagai2000,Nagai1997}).
As the results that solutions are not always existing globally in time, let $T_{ \text{max}}$ be the maximal existence time of solutions and in what follows we shall take $T\in (0,T_{\text{max}})$.
Now we begin with basic properties of the general KS model \eqref{2.1}-\eqref{1.4}, and provide detailed proofs which will be useful later in deriving similar properties for the  scheme \eqref{S1}-\eqref{S2}.

\begin{theorem}
The general KS model \eqref{2.1}-\eqref{1.4}  on $\Omega\times[0,T]$ satisfies the following properties:
\begin{enumerate}
  \item Mass conservation:
  \begin{equation}\label{2.4}
  \int_{\Omega} \rho(\boldsymbol{x}, t) d \boldsymbol{x}=\int_{\Omega} \rho_0(\boldsymbol{x}) d \boldsymbol{x}.
  \end{equation}

  \item Positivity preserving:
  if $\rho_0(\boldsymbol{x}), c_0(\boldsymbol{x})\geq 0$  and $\rho_0, c_0(\boldsymbol{x}) \not\equiv 0$, then $\rho(\boldsymbol{x},t)\geq 0$ and $c(\boldsymbol{x},t)\geq 0$;
      if moreover $\rho_0(x)>0$ then $\rho(\boldsymbol{x},t) > 0$ and $c(\boldsymbol{x},t) > 0$.

  \item Energy dissipation:
    \begin{equation}\label{2.6}
  \frac{d E_{tot}(\rho,c)}{dt}=
  - \int_{\Omega} \rho  |\nabla \left( \log \rho-\chi c \right)|^{2}
  d\boldsymbol{x} - \frac{\tau \chi}{\gamma} \int_{\Omega} \left( \frac{\partial c}{\partial t}\right)^2\leq 0,
  \end{equation}

  where the free energy of the system \eqref{2.1}-\eqref{1.4} is defined by
  \begin{equation}\label{2.7}
  E_{tot}(\rho,c)
  =\int_{\Omega}\left( f(\rho)
  - \chi \rho c
  + \frac{\chi}{2\gamma}|\nabla c|^{2}
  + \frac{\alpha \chi}{2\gamma}c^{2} \right) d\boldsymbol{x},
  \end{equation}
  with $f(\rho) = \rho \log \rho -\rho$.
\end{enumerate}
\end{theorem}

\begin{proof}
Integrating the equation \eqref{2.1} over $\Omega$, we deduce the mass conservation
\begin{equation}\label{2.11}
\frac{d}{d t} \int_{\Omega} \rho(\boldsymbol{x}, t) d\boldsymbol{x}=0,
\end{equation}
with either periodic or Neumann boundary condition for $\rho$.

Set $\rho_{+}:=\text{sup}\{\rho,0 \}$, $\rho_{-}:=\text{sup}\{-\rho,0 \}$, then $\rho = \rho_{+}-\rho_{-}$ and $|\rho| = \rho_{+}+\rho_{-}$.
Multiplying both sides of the equation \eqref{2.1} by the sign function $\text{sgn} \rho$, we have:
\begin{equation}\label{2.8}
\frac{\partial \rho}{\partial t} \text{sgn} \rho
= \Delta \rho \,\text{sgn} \rho - \chi\nabla \cdot \left(\rho \nabla c \right) \text{sgn} \rho.
\end{equation}
Recalling that $\text{sgn} \xi = \frac{\xi}{|\xi|}$ and $\xi \text{sgn} \xi= |\xi|$, we have
\begin{align*}
\frac{\partial \rho}{\partial t} \text{sgn} \rho=\frac{\partial |\rho|}{\partial t}, \ \
\nabla \cdot (\rho \nabla c) \text{sgn} \rho = \nabla \cdot (|\rho| \nabla c).
\end{align*}
By the Kato's inequality \cite{Kato},
we know that $\Delta \rho \, \text{sgn}\rho \leq \Delta |\rho|$ in $D'(\Omega)$.
Then, we have
\begin{equation}\label{2.9}
\frac{\partial |\rho|}{\partial t}\leq \Delta |\rho| - \chi\nabla \cdot (|\rho| \nabla c).
\end{equation}
Integrating both sides of the inequality over the region $\Omega$ leads to
\begin{equation}\label{2.10}
\frac{d}{d t}\int_{\Omega} |\rho| d\boldsymbol{x}
= \frac{d}{d t}\int_{\Omega} (\rho_{+}+\rho_{-}) d\boldsymbol{x} \leq 0,
\end{equation}
with periodic or Neumann boundary condition.
On the other hand, by the mass conservation, we have
\begin{equation}\label{2.10-1}
\frac{d}{d t} \int_{\Omega} \rho d\boldsymbol{x}
= \frac{d}{d t}\int_{\Omega} (\rho_{+} - \rho_{-}) d\boldsymbol{x}=0.
\end{equation}
Subtracting the above from \eqref{2.10},
we derive
\begin{equation}\label{2.11-1}
\frac{d}{d t} \int_{\Omega} \rho_{-} d\boldsymbol{x} \leq 0.
\end{equation}
Therefore,
 we have $\rho_{-}=0$ if $\rho_{0_{-}}=0$, which implies that $\rho(\boldsymbol{x},t)\geq 0$ if the initial data $\rho_{0} \geq 0$.

Moreover if $\rho_0>0$ we can prove that $\rho(\boldsymbol{x},t)>0$ as follows.
Fix $\Omega\times[0,T]$ a given cylinder.
Fix $\varepsilon>0$ small enough such that $\rho_0>\varepsilon$.
Consider the function $v(\boldsymbol{x},t)=\rho(\boldsymbol{x},t)-\varepsilon \exp(-\lambda t).$
We have
\begin{equation}\label{2.11-add}
\frac{\partial v}{\partial t}-\Delta v+\chi \nabla.(v\nabla c)
= \varepsilon \exp(-\lambda t) \left(\lambda- \chi \Delta c\right).
\end{equation}
Using the basic theory about linear elliptic and linear parabolic equations in Evans' book \cite{PDE}, we carry out the following analysis.
In the parabolic-elliptic case $\tau=0$, we have that $-\Delta c \geq -\alpha c$.
Moreover, we know that $\rho$ remains bounded in $L^\infty(\Omega)$
(see Section \ref{bounds} below).
Then $c$ is smooth  and the right hand side
of \eqref{2.11-add} is bounded from below by $\varepsilon \exp(-\lambda t) \left(\lambda- \chi \alpha \| c\|_{L^\infty}\right).$
In the parabolic-parabolic case $\tau>0$, assuming that $c_0$ is smooth enough,
we have that $\|-\Delta c\|_{L^\infty}$ is bounded and the right hand side
of \eqref{2.11-add} is bounded from below by $\varepsilon \exp(-\lambda t) \left(\lambda- \chi \|-\Delta c\|_{L^\infty}\right).$
Choosing $\lambda$ large enough such that the right hand side of this inequality is non negative and proceeding as above, we have that $\rho(\boldsymbol{x},t)$ is bounded  below by a positive function on $\Omega\times[0,T]$.

Next we prove the positivity of $c$ in the case $\tau=0$.
We derive from  \eqref{2.2} that
\begin{equation}\label{2.11-2}
c = \gamma \left( -\Delta+\alpha I\right)^{-1} \rho,
\end{equation}
where $I$ is the identity matrix and the operator $-\Delta+\alpha I$ is positive \cite{PDE}.
Therefore, we derive the positivity of $c(\boldsymbol{x},t)$ since we already showed the positivity of $\rho(\boldsymbol{x},t)$.

In the case $\tau>0$ we can use an analogous argument.
In fact the parabolic equation is the limit of the sequence of following elliptic equations
\begin{equation}\label{2.11-3}
\tau\frac{c^{n+1}-c^{n}}{\delta t} = \Delta c^{n+1} - \alpha c^{n+1} + \gamma \rho^{n+1}, \ \ \ \forall n \in \mathbb{N},
\end{equation}
which satisfy the maximum principle. Then the limit parabolic equation also satisfy the maximum principle.

Now we consider the energy dissipation.
Since  $\Delta \rho=\nabla \cdot(\rho \nabla \log \rho)$,
we rewrite \eqref{2.1} as
\begin{align}\label{2.13}
\frac{\partial \rho}{\partial t}
= \nabla \cdot
\left( \rho \nabla \log \rho \right)- \chi \nabla \cdot \left(\rho \nabla c \right) = \nabla \cdot \left( \rho\nabla
\left( \log \rho - \chi c \right) \right).
\end{align}
According to the free energy \eqref{2.7}, we know the energy derivative
$\frac{\delta E_{tot}}{\delta \rho}=\log \rho - \chi c$ and
$\frac{\delta E_{tot}}{\delta c}=-\frac{\chi}{\gamma}\Delta c + \frac{\alpha \chi}{\gamma} c - \chi\rho=-\frac{\tau \chi}{\gamma}\frac{\partial c}{\partial t}$.
Taking the inner product of \eqref{2.13} with $\frac{\delta E_{tot}}{\delta \rho}$ and \eqref{2.2} with $\frac{\chi}{\gamma}\frac{\partial c}{\partial t}$ followed by integration by parts, we obtain the energy dissipative law \eqref{2.6},
since the boundary terms vanish thanks to either to periodic or Neumann boundary conditions.

\end{proof}

\section{Bounds of the exact solution}\label{bounds}
In this section, we derive $L^p$-bounds $(1<p \leq \infty)$ of the exact solution $\rho$
of the general KS model \eqref{2.1}-\eqref{1.4} on $\Omega\times[0,T]$.
Note that $L^\infty$  bound was derived in  \cite{C2004}  in the whole space $\mathbf{R}^2$. However,  the proof is out of reach here since it depends on  Nash's inequality that is not valid for instance in periodic setting \cite{Nash}.
As for the periodic case, the authors in \cite{Horstmann2001,Nagai1997} used Trudinger-Moser inequality under the small mass condition of  $\int_{\Omega} \rho_0 d\boldsymbol{x} < \frac{4 \pi}{\chi\gamma}$ to derive uniform estimates for $\tau>0$.
The case with $\tau=0$ is considered in \cite{Biler1993,Biler1995}.

We will use frequently various continuous and discrete Gronwall lemmas.
In addition to the   classical Gronwall lemmas (see, for instance,  \cite{Emmrich}), we also need  the uniform Gronwall lemma in \cite{Temam}, which is stated here:
\begin{lemma}\label{ug}
Consider $y,g,h$ nonnegative functions such that
$$ \frac{dy}{dt} \leq g y+h, $$
\noindent and for any $t>0$
$$ \int_t^{t+1} y(s)ds\leq k_1, \; \int_t^{t+1} g(s)ds\leq k_2, \; \int_t^{t+1} h(s)ds\leq k_3. $$
\noindent Then for any $t\geq 1$ we have $y(t)\leq (k_1+k_3)\exp(k_2).$
\end{lemma}

We derive below a  series of bound results about the exact solution $\rho$ in $L^p(\Omega)$ $(1<p \leq \infty)$ for the general KS model \eqref{2.1}-\eqref{1.4}.

\begin{lemma}\label{estintegral}
For $T\in (0,T_{\text{max}})$, if $(2+\tau)\gamma \chi C_{gn}M< 1$,
then $\rho(\boldsymbol{x}, t) \in L^{2}\left( 0,T; L^2(\Omega) \right)$ for the general KS model \eqref{2.1}-\eqref{1.4},
where $M = \int_{\Omega} \rho(\boldsymbol{x}, t) d\boldsymbol{x}$ is mass of the cell density and $C_{gn}$ is the positive constant coefficient of Gagliardo-Nirenberg inequality in 2D  domain $\Omega$.
\end{lemma}
\begin{proof}
To begin with, we expand
\begin{align}\label{2.21}
\int_{\Omega} \rho  |\nabla \left( \log \rho-\chi c \right)|^{2} d\boldsymbol{x}
= \int_{\Omega} \frac{| \nabla \rho|^2}{\rho} d\boldsymbol{x}
+ \chi^2 \int_{\Omega}  \rho | \nabla c|^2 d\boldsymbol{x}
- 2 \chi \int_{\Omega} \nabla \rho \nabla c d\boldsymbol{x}.
\end{align}

According to the equation \eqref{2.2} about $c$, we have
\begin{align}\label{2.22}
\Delta c  = \tau \frac{\partial c}{\partial t} + \alpha c - \gamma\rho.
\end{align}
Combining with the Green formula and $\rho,c> 0$, we can estimate the last term in \eqref{2.21} as follows
\begin{equation}\label{2.23}
\begin{aligned}
- 2 \chi \int_{\Omega} \nabla \rho \nabla c d\boldsymbol{x}
&= 2 \chi \int_{\Omega} \rho \Delta c d\boldsymbol{x}  \\
&= 2 \chi \int_{\Omega} \rho \left( \tau \frac{\partial c}{\partial t} + \alpha c - \gamma\rho \right) d\boldsymbol{x}  \\
&\geq 2 \tau \chi \int_{\Omega} \rho \frac{\partial c}{\partial t}  d\boldsymbol{x}
- 2 \chi \gamma \| \rho \|^2_{L^2} .
\end{aligned}
\end{equation}
By Cauchy-Schwarz inequality, we know that
\begin{equation}\label{kmg1}
2 \tau \chi \int_{\Omega} \rho \frac{\partial c}{\partial t}  d\boldsymbol{x}
\geq -\left| 2 \tau \chi \int_{\Omega} \rho \frac{\partial c}{\partial t}  d\boldsymbol{x} \right|
\geq -\frac{\tau \chi}{\gamma} \int_\Omega \left(\frac{\partial c}{\partial t}\right)^2  d\boldsymbol{x} -  \tau\chi \gamma \| \rho \|^2_{L^2}.
\end{equation}
Applying the Gagliardo-Nirenberg inequality with a positive constant $C_{gn}$
\begin{equation}\label{GN}
\left\|\sqrt \rho \right\|^4_{L^4}
\leq C_{gn} \left\|\sqrt \rho \right\|^2_{L^2}\left(\left\|\sqrt \rho \right\|^2_{L^2}+\left\|\nabla \sqrt \rho \right\|^2_{L^2} \right),
\end{equation}
we find
\begin{equation}\label{2.24-0}
\begin{aligned}
-  \| \rho \|^2_{L^2}
= -  \left\|\sqrt \rho \right\|^4_{L^4}
\geq  -   C_{gn} M(M+ \left\|\nabla \sqrt \rho \right\|^2_{L^2})
=  -   C_{gn} M \left(M+ \int_{\Omega} \frac{| \nabla \rho|^2}{\rho} d\boldsymbol{x}\right),
\end{aligned}
\end{equation}
where the mass $M = \int_{\Omega} \rho(\boldsymbol{x}, t) d\boldsymbol{x}$.
Therefore, \eqref{2.21} can be estimated as
\begin{equation}\label{2.24}
\begin{aligned}
&\int_{\Omega} \rho  |\nabla \left( \log \rho-\chi c \right)|^{2} d\boldsymbol{x}+ \frac{\tau \chi}{\gamma} \int_\Omega \left(\frac{\partial c}{\partial t} \right)^2  d\boldsymbol{x}
+ (2+\tau) \chi \gamma C_{gn} M^2 \\
\geq & \left( 1-(2+\tau) \chi \gamma C_{gn} M \right)
\int_{\Omega} \frac{| \nabla \rho|^2}{\rho} d\boldsymbol{x}
+ \chi^2 \int_{\Omega}  \rho | \nabla c|^2 d\boldsymbol{x}.
\end{aligned}
\end{equation}

By the Green formula and Young's inequality, we get
\begin{align}\label{2.25}
-\frac{1}{\gamma} \int_{\Omega} \rho \Delta c d\boldsymbol{x}
=\frac{1}{\gamma} \int_{\Omega} \nabla \rho \nabla c d\boldsymbol{x}
\leq \frac{1}{2\gamma} \int_{\Omega} \frac{| \nabla \rho|^2}{\rho} d\boldsymbol{x}
+ \frac{1}{2\gamma} \int_{\Omega}  \rho | \nabla c|^2 d\boldsymbol{x}.
\end{align}
Hence, we follow the equation \eqref{2.2} and the positivity of solutions to obtain
\begin{equation}\label{2.26}
\begin{aligned}
\| \rho \|^2_{L^2}
\leq&-\frac{1}{\gamma} \int_{\Omega} \rho \Delta c d\boldsymbol{x}
+ \frac{\tau}{\gamma}\int_{\Omega} \rho\frac{\partial c}{\partial t} d\boldsymbol{x}\\
\leq&\frac{1}{2\gamma} \int_{\Omega} \frac{| \nabla \rho|^2}{\rho} d\boldsymbol{x}
+ \frac{1}{2\gamma} \int_{\Omega}  \rho | \nabla c|^2 d\boldsymbol{x}
+ \frac{\tau}{\gamma}\int_{\Omega} \rho\frac{\partial c}{\partial t} d\boldsymbol{x}.
\end{aligned}
\end{equation}
Moreover, appealing Cauchy-Schwarz inequality
\begin{equation}\label{2.27}
\frac12 \| \rho \|^2_{L^2}
\leq \frac{1}{2\gamma} \int_{\Omega} \frac{| \nabla \rho|^2}{\rho} d\boldsymbol{x}
+ \frac{1}{2\gamma} \int_{\Omega}  \rho | \nabla c|^2 d\boldsymbol{x}+ \frac{\tau^2}{2\gamma^2}\int_{\Omega} \left(\frac{\partial c}{\partial t} \right)^2 d\boldsymbol{x}.
\end{equation}

Gathering inequalities \eqref{2.24} and \eqref{2.27},  we infer that if $(2+\tau)\chi\gamma C_{gn}M< 1$,
then for some constant $C$
\begin{equation}\label{2.28}
\| \rho \|^2_{L^2} \leq C \left(\int_{\Omega} \rho  |\nabla \left( \log \rho-\chi c \right)|^{2} d\boldsymbol{x}+ \frac{\tau \chi}{\gamma} \int_\Omega \left(\frac{\partial c}{\partial t} \right)^2  d\boldsymbol{x}
+ (2+\tau) \chi \gamma C_{gn} M^2 \right),
\end{equation}
that is integrable in $[0,T]$, due to the energy dissipation \eqref{2.6}.
\end{proof}

\begin{lemma}\label{estL2}
If initial data $\rho_0 \in L^2(\Omega)$ and Lemma \ref{estintegral} holds, then
$\rho(\boldsymbol{x},t) \in L^{\infty}\left( 0,+\infty; L^2(\Omega) \right)$
for the general KS model \eqref{2.1}-\eqref{1.4}.
\end{lemma}

\begin{proof}
Taking the scalar product of \eqref{2.1} with $\rho$ leads to
\begin{align}\label{2.31}
\left( \frac{\partial \rho}{\partial t}, \rho \right)
-\left( \Delta \rho, \rho \right) = - \chi \left( \nabla \cdot \left(\rho \nabla c \right). \rho \right),
\end{align}
By the Green formula, the equation \eqref{2.2} about $c$ and the positivity of $c$, we have
\begin{align}\label{2.32}
\left( \frac{\partial \rho}{\partial t}, \rho \right)
= \frac{1}{2}\frac{d}{dt} \| \rho \|^2_{L^2},
\ \ \
-\left( \Delta \rho, \rho \right) = \| \nabla\rho \|^2_{L^2},
\end{align}
and

\begin{align}\label{2.34}
- \chi \left( \nabla \cdot \left(\rho \nabla c \right) , \rho \right)
&= \chi \left( \rho \nabla c , \nabla \rho \right)
= \frac{\chi}{2} \left( \rho^2,  - \Delta c \right) \nonumber\\
&= \frac{\chi}{2} \left( \rho^2, \gamma \rho - \tau\frac{\partial c}{\partial t} - \alpha c \right) \nonumber\\
&\leq \frac{\chi \gamma}{2} \left( \rho^2, \rho \right)
- \frac{\tau\chi}{2} \left( \rho^2, \frac{\partial c}{\partial t} \right).
\end{align}

For the two terms on the right-hand side of the inequality \eqref{2.34}, we apply Cauchy-Schwarz inequality, Gagliardo-Nirenberg inequality and Young's inequality in turn to derive
\begin{align}\label{2.35}
\frac{\chi \gamma}{2} \left( \rho^2, \rho \right)
&\leq \frac{\chi \gamma}{2} \|\rho^2\|_{L^2} \|\rho\|_{L^2}\nonumber\\
&\leq \frac{\chi \gamma C_{gn}}{2} \|\rho\|^2_{L^2} (\|\nabla\rho\|_{L^2}+\|\rho\|_{L^2})\nonumber\\
&\leq \frac{\chi \gamma C_{gn}}{2}
\left(\frac 1 \varepsilon  \|\rho\|^4_{L^2}
+ \varepsilon \|\nabla\rho\|^2_{L^2}
+ \frac{1}{2}\|\rho\|^4_{L^2}
+ \frac{1}{2} \|\rho\|^2_{L^2} \right)\nonumber\\
&= \frac{\chi \gamma C_{gn}}{2} \left(\frac{1}{2}+ (\frac 1 \varepsilon +\frac 1 2 )\|\rho\|^2_{L^2}\right)\|\rho\|^2_{L^2}
+ \frac{\chi \gamma C_{gn}}{2} \varepsilon  \|\nabla\rho\|^2_{L^2},
\end{align}
and

\begin{align}\label{2.35add}
-\frac{\tau\chi}{2} \left( \rho^2, \frac{\partial c}{\partial t} \right)
&\leq | -\frac{\tau\chi}{2} \left( \rho^2, \frac{\partial c}{\partial t} \right) |
\leq \frac{\tau\chi}{2}  \left\|\frac{\partial c}{\partial t} \right\|_{L^2} \left\| \rho^2 \right\|_{L^2} \nonumber\\
&\leq \frac{\tau\chi C_{gn}}{2}  \left\|\frac{\partial c}{\partial t} \right\|_{L^2}
\left( \| \nabla \rho\|_{L^2} \| \rho\|_{L^2} + \| \rho\|_{L^2}^2 \right) \nonumber\\
&\leq \frac{\tau\chi C_{gn}}{2}
\left( \frac 1 \varepsilon \left\|\frac{\partial c}{\partial t} \right\|_{L^2}^2 \|\rho\|_{L^2}^2
+ \varepsilon  \| \nabla \rho\|_{L^2}^2
+ \left\|\frac{\partial c}{\partial t} \right\|_{L^2}\| \rho\|_{L^2}^2
\right) \nonumber\\
&= \frac{\tau\chi C_{gn}}{2} \left( \frac 1 \varepsilon \left\|\frac{\partial c}{\partial t} \right\|_{L^2}^2
+\left\|\frac{\partial c}{\partial t} \right\|_{L^2} \right) \|\rho\|_{L^2}^2
+ \frac{\tau\chi C_{gn}}{2} \varepsilon \| \nabla \rho\|_{L^2}^2.
\end{align}

We substituting \eqref{2.32}-\eqref{2.35add} into \eqref{2.31} to obtain
\begin{align}\label{2.36}
&\frac{1}{2}\frac{d}{dt} \| \rho \|^2_{L^2}
+ \| \nabla\rho \|^2_{L^2} \nonumber\\
\leq & \frac{\chi C_{gn}}{2}
\left(\gamma+2\gamma\|\rho\|^2_{L^2} + \tau \left\|\frac{\partial c}{\partial t} \right\|_{L^2}^2 +\tau\left\|\frac{\partial c}{\partial t} \right\|_{L^2} \right)
\|\rho\|^2_{L^2}
+ \frac{\chi C_{gn} \varepsilon  }{2} \left( \gamma+\tau \right)
\|\nabla\rho\|^2_{L^2}.
\end{align}
As long as $\varepsilon$ is small enough so that $\frac{\chi C_{gn} }{2}(\gamma+\tau)\varepsilon\leq 1$, we have
\begin{align}\label{2.37}
\frac{d}{dt} \| \rho \|^2_{L^2}
\leq   C  \left(1 + \|\rho\|^2_{L^2} + \tau \left\|\frac{\partial c}{\partial t} \right\|_{L^2}^2 +\tau \left\|\frac{\partial c}{\partial t} \right\|_{L^2} \right) \|\rho\|^2_{L^2}
\end{align}
by integrating $\frac 1 \varepsilon$ into the constant $C$.
Due to the estimate \eqref{2.28} and energy dissipation \eqref{2.6}, we have that
$$
\int_t^{t+1} \| \rho \|^2_{L^2}\leq E_{tot}(\rho_0,c_0)+ (2+\tau) \chi \gamma C_{gn} M^2=K_0.
$$
Besides,
$$
\int_t^{t+1} \left(1 + \|\rho\|^2_{L^2} + \tau \left\|\frac{\partial c}{\partial t} \right\|_{L^2}^2 +\tau \left\|\frac{\partial c}{\partial t} \right\|_{L^2} \right)
\leq C\left(1+ E_{tot}(\rho_0,c_0)+ (2+\tau) \chi \gamma C_{gn} M^2 \right)
=C(1+K_0).
$$
For $t\geq 1$, applying the uniform Gronwall Lemma \ref{ug} leads to a bound on $\| \rho(t) \|^2_{L^2}$ that reads $K_0\exp(C(1+K_0))$.
For $t< 1$, we apply the classical Gronwall lemma to obtain
$\| \rho \|^2_{L^2} \leq \tilde K_0 < +\infty$, where
$$
\tilde K_0:=\|\rho_0 \|^2_{L^2} \exp \left( \int_0^1   C  \left(1 + \|\rho\|^2_{L^2} + \tau \left\|\frac{\partial c}{\partial t} \right\|_{L^2}^2 +\tau \left\|\frac{\partial c}{\partial t} \right\|_{L^2} \right) dt \right)
=\|\rho_0 \|^2_{L^2} \exp\left(C(1+K_0) \right),
$$
is a constant followed by Lemma \ref{estintegral}.
As a result, we have
\begin{align}\label{2.38}
\| \rho(t) \|^2_{L^2} \leq K_1 :=\max \{K_0, \|\rho_0 \|^2_{L^2} \} \exp\left(C(1+K_0)\right).
\end{align}
\end{proof}

\begin{lemma}\label{estLp}
For any $1<p\leq +\infty$ and $T\in (0,T_{\text{max}})$, if initial data $\rho_0 \in L^p(\Omega)$ and Lemma \ref{estL2} holds, then $\rho(\boldsymbol{x},t) \in L^{\infty}\left( 0,T; L^p(\Omega) \right)$ for the general KS model \eqref{2.1}-\eqref{1.4}.
\end{lemma}

\begin{proof}
Taking the scalar product of \eqref{2.1} with $\rho^{2p-1}$ leads to
\begin{align}\label{2.41}
\left( \frac{\partial \rho}{\partial t}, \rho^{2p-1} \right)
-\left( \Delta \rho, \rho^{2p-1} \right)
= - \chi \left( \nabla \cdot \left(\rho \nabla c \right) , \rho^{2p-1} \right),
\end{align}
where
\begin{align}\label{2.42}
\left( \frac{\partial \rho}{\partial t}, \rho^{2p-1} \right)
= \frac{1}{2p}\frac{d}{dt} \left\| \rho^{p} \right\|^2_{L^2}.
\end{align}
\begin{align}\label{2.43}
- \left( \Delta \rho, \rho^{2p-1} \right) = \left( \nabla \rho, \nabla\rho^{2p-1} \right)
= (2p-1) \| \rho^{p-1}\nabla \rho\|^2_{L^2}
= \frac{2p-1}{p^2} \left\| \nabla (\rho^{p}) \right\|^2_{L^2}.
\end{align}
Besides, proceeding as above
\begin{align}\label{2.44}
- \chi \left( \nabla \cdot \left(\rho \nabla c \right) , \rho^{2p-1} \right)
= \chi \left( \rho \nabla c , \nabla \rho^{2p-1} \right)
= \frac{(2p-1)\chi}{2p} \left(\nabla (\rho^{2p}), \nabla c \right)
= \frac{(2p-1)\chi}{2p} \left( \rho^{2p},  - \Delta c \right).
\end{align}

By the equation \eqref{2.2} about $c$ and the positivity of $\rho$ and $c$,
\begin{align}\label{2.45}
\left( \rho^{2p},  - \Delta c \right)
=\left( \rho^{2p},  \gamma \rho - \tau\frac{\partial c}{\partial t} - \alpha c  \right)
\leq \gamma \left( \rho^{2p}, \rho \right)
- \tau \left( \rho^{2p}, \frac{\partial c}{\partial t} \right).
\end{align}
Applying Cauchy-Schwarz inequality, Gagliardo-Nirenberg inequality and Young's inequality in turn, we derive that
for some $\varepsilon>0$ we have
\begin{align}\label{2.45-1}
\gamma \left( \rho^{2p}, \rho \right)
&\leq \gamma \|\rho\|_{L^2} \left\|\rho^{2p} \right\|_{L^2} \nonumber\\
&\leq \gamma \|\rho\|_{L^2} C_{gn} \|\rho^p\|_{L^2} \left( \|\nabla(\rho^p)\|_{L^2}+ \|\rho^p\|_{L^2} \right) \nonumber\\
&\leq  \frac{\gamma C_{gn}}{\varepsilon} \|\rho\|_{L^2}^2 \|\rho^p\|_{L^2}^2
+ \gamma C_{gn} \varepsilon \|\nabla(\rho^p)\|_{L^2}^2
+\gamma C_{gn} \|\rho\|_{L^2} \|\rho^p\|_{L^2}^2 \nonumber\\
&\leq \gamma C_{gn} \left( \frac{K_1}{\varepsilon} + K_1^{\frac{1}{2}} \right)\|\rho^p\|_{L^2}^2
+ \gamma C_{gn} \varepsilon \|\nabla(\rho^p)\|_{L^2}^2,
\end{align}
where $\|\rho\|^2_{L^2} \leq K_1$ is known in Lemma 3.3. Similarly,
\begin{align}\label{2.45-2}
-\tau \left( \rho^{2p}, \frac{\partial c}{\partial t} \right)
&\leq \left|-\tau \left( \rho^{2p}, \frac{\partial c}{\partial t} \right) \right|
\leq\tau \left\|\frac{\partial c}{\partial t} \right\|_{L^2} \left\|\rho^{2p} \right\|_{L^2} \nonumber\\
&\leq \tau \left\|\frac{\partial c}{\partial t} \right\|_{L^2} C_{gn} \left\|\rho^p \right\|_{L^2} \left( \|\nabla(\rho^p)\|_{L^2}+ \|\rho^p\|_{L^2} \right) \nonumber\\
&\leq \tau C_{gn} \left( \frac{1}{\varepsilon} \left\|\frac{\partial c}{\partial t} \right\|_{L^2}^2 + \left\|\frac{\partial c}{\partial t} \right\|_{L^2} \right)\|\rho^p\|_{L^2}^2
+ \tau C_{gn}  \varepsilon \|\nabla(\rho^p)\|_{L^2}^2.
\end{align}

Substituting \eqref{2.42}-\eqref{2.45-2} into \eqref{2.41}, we obtain
\begin{align}\label{2.46}
&\frac{1}{2p}\frac{d}{dt} \| \rho^{p} \|^2_{L^2}
+ \frac{2p-1}{p^2} \| \nabla (\rho^{p})\|^2_{L^2} \nonumber\\
\leq &\frac{(2p-1)\chi C_{gn}}{2p} \left( \gamma K_1 + \gamma K_1^{\frac{1}{2}}
+\tau \left\|\frac{\partial c}{\partial t} \right\|_{L^2}^2 + \tau \left\|\frac{\partial c}{\partial t} \right\|_{L^2} \right)
\|\rho^p\|^2_{L^2}  \nonumber\\
+& \frac{(2p-1)\chi C_{gn}}{2p} (\gamma+ \tau) \varepsilon
\|\nabla(\rho^p)\|^2_{L^2}.
\end{align}
As long as $\varepsilon$ is small enough so that $\frac{\chi p C_{gn}}{2} (\gamma + \tau) \varepsilon\leq 1$, we have
\begin{align}\label{2.47}
\frac{d}{dt} \| \rho^{p} \|^2_{L^2} = \frac{d}{dt} \| \rho \|^{2p}_{L^{2p}}
\leq C p \left( \frac{K_1}{\varepsilon}   + K_1^{\frac{1}{2}}
+\tau \left\|\frac{\partial c}{\partial t} \right\|_{L^2}^2 + \frac{\tau}{\varepsilon} \left\|\frac{\partial c}{\partial t} \right\|_{L^2} \right) \| \rho \|^{2p}_{L^{2p}}.
\end{align}
According to the Gronwall lemma and Lemma 3.2, we have
\begin{align}\label{2.48}
\| \rho \|^{2p}_{L^{2p}} \leq K_2 < +\infty,\ \ \forall 1<p<+\infty,
\end{align}
where
\begin{align}\label{2.49}
K_2:=\| \rho_0 \|^{2p}_{L^{2p}} \exp \left(  C  p ( K_1 T + K_0) \right),
\end{align}
with $\frac 1 \varepsilon$ being absorbed by the constant $C$.

We can then obtain the $L^\infty$-norm estimate by
$\| \rho \|_{L^{2p}}\rightarrow \|\rho\|_{L^\infty}$ when $p\rightarrow +\infty$.

\end{proof}

\section{Properties of the semi-discrete scheme}
Since it is already shown in  \cite{Saito2005} that the semi-implicit Euler scheme \eqref{S1}-\eqref{S2} has unique solutions if $\delta t$ is sufficiently small,
 we show below that the  scheme  \eqref{S1}-\eqref{S2} satisfies the mass conservation, positivity preserving, and  energy dissipation.

\begin{theorem}
Let $N_t \in \mathbb{N}$,  $N_t \geq 1$, $\delta t = T/N_t$ and $t^{n+1}=(n+1) \delta t \leq T$ for $n\in \{0,1,2,\cdots,N_t-1\}$,
the semi-discrete scheme \eqref{S1}-\eqref{S2} satisfies the following properties:
\begin{enumerate}
  \item Mass conserving: $\int_{\Omega} \rho^{n+1} d\boldsymbol{x} = \int_{\Omega} \rho^{n} d\boldsymbol{x}$.

  \item Positivity preserving:
  if $\rho^{n},c^n \geq 0$, then $\rho^{n+1}\geq 0$ and $c^{n+1}\geq 0$;
  if moreover $\rho^{n}> 0$, then $\rho^{n+1} > 0$ and $c^{n+1} > 0$.

  \item Energy dissipation:
  \begin{align}\label{e03}
  &E_{tot}(\rho^{n+1},c^{n+1})-E_{tot}(\rho^{n},c^{n}) \leq 0,
  \end{align}
  where
  $E_{tot}(\rho^{n},c^{n})
  =\int_{\Omega}\left( f(\rho^{n})
  - \chi \rho^{n} c^{n}
  + \frac{\chi}{2\gamma}|\nabla c^{n}|^{2}
  + \frac{\alpha \chi}{2\gamma} (c^{n})^{2} \right) d\boldsymbol{x}$.
\end{enumerate}
\end{theorem}

\begin{proof}
Integrating the equation \eqref{S1} over $\Omega$, we deduce
\begin{equation}\label{3.1}
\int_{\Omega} \rho^{n+1} d\boldsymbol{x} = \int_{\Omega} \rho^{n} d\boldsymbol{x},
\end{equation}
with periodic or Neumann boundary condition for $\rho^{n+1}$.

Set $\rho^{n}_{+}:=\text{sup}\{\rho^{n},0 \}$, $\rho^{n}_{-}:=\text{sup}\{-\rho^{n},0 \}$, then $\rho^{n} = \rho^{n}_{+} - \rho^{n}_{-}$ and $|\rho^{n}| = \rho^{n}_{+} + \rho^{n}_{-}$.
Multiplying both sides of the equation \eqref{S1} by the sign function $\text{sgn} \rho^{n+1}$, we have:
\begin{equation}\label{3.2}
\frac{\rho^{n+1}\text{sgn} \rho^{n+1}-\rho^{n}\text{sgn} \rho^{n+1}}{\delta t} =
 \Delta \rho^{n+1} \text{sgn} \rho^{n+1}- \chi\nabla \cdot \left(\rho^{n+1} \nabla c^{n} \right)\text{sgn} \rho^{n+1}.
\end{equation}
By the property of symbolic function and the Kato's inequality,
we have
\begin{equation}\label{3.3}
|\rho^{n+1}| \leq |\rho^{n}| + \delta t \Delta |\rho^{n+1}| - \delta t\chi\nabla \cdot (|\rho^{n+1}| \nabla c^n).
\end{equation}
Integrating both sides of the inequality over the region $\Omega$, we obtain
\begin{equation}\label{3.4}
\int_{\Omega}|\rho^{n+1}| d\boldsymbol{x} = \int_{\Omega}\left( \rho^{n+1}_{+} + \rho^{n+1}_{-}\right) d\boldsymbol{x}
\leq
\int_{\Omega}|\rho^{n}| d\boldsymbol{x} = \int_{\Omega}\left( \rho^{n}_{+} + \rho^{n}_{-}\right) d\boldsymbol{x},
\end{equation}
where we used periodic or Neumann boundary condition for $\rho^{n+1}$.
Recall the mass conservation
\begin{equation}\label{3.4*}
\int_{\Omega} \rho^{n+1} d\boldsymbol{x} = \int_{\Omega}\left( \rho^{n+1}_{+} - \rho^{n+1}_{-}\right) d\boldsymbol{x}
= \int_{\Omega} \rho^{n} d\boldsymbol{x} = \int_{\Omega}\left( \rho^{n}_{+} - \rho^{n}_{-}\right) d\boldsymbol{x}.
\end{equation}
Subtracting the above from \eqref{3.4}, we derive
\begin{equation}\label{3.5}
\int_{\Omega} \rho^{n+1}_{-} d\boldsymbol{x} \leq \int_{\Omega} \rho^{n}_{-} d\boldsymbol{x}, \ \ \text{and}\ \ \rho^{n+1}_{-}=\rho^n_{-}=0,
\end{equation}
due to $\rho^{0} = \rho_{0} \geq 0$ and $\rho_{0}\not\equiv 0$.
Therefore $\rho^{n+1} \geq 0$.

To check that $\rho^{n+1}>0$ in $\Omega$ when $\rho_{0} > 0$ we proceed as follows.
Since $\rho_{0}$ belongs to $L^\infty$, then $\rho^1$ is smooth. Assume that $c_0$ is smooth.
If there exists a $x_0$ in $\Omega$ such that $\rho^1(x_0)=0$, then $\nabla \rho^1(x_0)=0$ and $\Delta \rho^1(x_0)\geq 0$.
Going back to \eqref{S1} then $\rho_0(x_0)\leq 0$, contradiction.
We conclude by induction on $n$.

According with the equation \eqref{S2} about $c^{n+1}$, we have in the elliptic case $\tau=0$,
\begin{equation}\label{3.6}
c^{n+1} = \gamma \left( -\Delta+\alpha I\right)^{-1} \rho^{n+1}.
\end{equation}
In the parabolic case $\tau>0$ the proof is similar as the above.
Therefore, the positivity of $c^{n+1}$ is consistent with that of $\rho^{n+1}$.

For the energy dissipation, we first rewrite \eqref{S1} as
\begin{align}\label{e4}
&\frac{\rho^{n+1}-\rho^{n}}{\delta t}= \nabla \cdot \left( \rho^{n+1} \nabla
\left(  \log\rho^{n+1} -  \chi c^{n} \right)
\right).
\end{align}
Taking the inner product of \eqref{e4} with $\delta t \left( \log\rho^{n+1} - \chi c^{n} \right)$, we get
\begin{align}\label{e5}
\int_{\Omega}\left( \rho^{n+1}-\rho^{n} \right)
\left( \log\rho^{n+1} - \chi c^{n} \right) d\boldsymbol{x}
= -\delta t \int_{\Omega} \rho^{n+1}
|\nabla\left( \log\rho^{n+1} - \chi c^{n} \right)|^2 d\boldsymbol{x}.
\end{align}
We note that $f(\rho^{n}) = \rho^{n} \log \rho^{n} -\rho^{n}$ is convex for $\rho^{n}>0$, then we can use $f'(\rho^{n})=\log\rho^{n}$ and derive
\begin{align}\label{e6}
\int_{\Omega}\left( \rho^{n+1}-\rho^{n} \right) f'(\rho^{n+1}) d\boldsymbol{x}
&=\int_{\Omega}\left( f(\rho^{n+1})- f(\rho^{n})+\frac{1}{2}(\rho^{n+1}-\rho^{n})^2  f''(\xi)  \right)d\boldsymbol{x} \nonumber\\
&\geq \int_{\Omega} \left( f(\rho^{n+1})- f(\rho^{n}) \right)d\boldsymbol{x}.
\end{align}
Taking the inner product of \eqref{S2} with $c^{n+1}-c^{n}$, we get
\begin{align}\label{e7}
&\frac{\tau}{\delta t} \|c^{n+1}-c^{n}\|^2_{L^2} + \frac{1}{2} \|\nabla c^{n+1}\|^2_{L^2}
- \frac{1}{2} \|\nabla c^{n}\|^2_{L^2} + \frac{1}{2} \|\nabla c^{n+1}-\nabla c^{n} \|^2_{L^2} \nonumber \\
+ &\frac{\alpha}{2} \|c^{n+1}\|^2_{L^2} - \frac{\alpha}{2} \|c^{n}\|^2_{L^2}  + \frac{\alpha}{2} \|c^{n+1}-c^{n}\|^2_{L^2} - \gamma (\rho^{n+1}, c^{n+1}-c^{n}) =0.
\end{align}
We sum up \eqref{e5}-\eqref{e7} and obtain the discrete energy dissipation
\begin{align}\label{e8}
&E_{tot}(\rho^{n+1},c^{n+1})-E_{tot}(\rho^{n},c^{n}) \nonumber \\
\leq &-\delta t \int_{\Omega} \rho^{n+1}
|\nabla\left( \log\rho^{n+1} - \chi c^{n} \right)|^2 d\boldsymbol{x}
- \frac{\chi}{2\gamma} \|\nabla c^{n+1}-\nabla c^{n} \|^2_{L^2}
\nonumber \\
& - \frac{\chi}{\gamma}(\frac{\tau}{\delta t}+ \frac{\alpha}{2})\|c^{n+1}-c^{n}\|^2_{L^2} \leq 0.
\end{align}

\end{proof}

\section{Bounds of the numerical solution}\label{boundsdiscrete}
To derive error estimates of the semi-implicit Euler scheme \eqref{S1}-\eqref{S2},
we need to estimate $L^p$-bounds $(1<p<\infty)$ of the semi-discrete numerical solution $\rho^{n+1}$ for $n\in \{0,1,2,\cdots,N_t-1\}$ in advance.

\begin{lemma}\label{destintegral}
If $(2+\tau) \gamma \chi C_{gn} M< 1$, then there exists $C>0$ such that $\delta t \sum\limits_{n=0}^{N_t-1} \left\|\rho^{n+1}\right\|^2_{L^2} \le C$ for the semi-discrete scheme \eqref{S1}-\eqref{S2}.
\end{lemma}

\begin{proof}
Similarly to the derivation in Lemma \ref{estintegral}, we have the discrete form
\begin{equation}\label{3.20-1}
\begin{aligned}
\left\| \rho^{n+1} \right\|^2_{L^2}
&\leq C \Big( \int_{\Omega} \rho^{n+1}  |\nabla \left( \log \rho^{n+1}-\chi  c^{n+1} \right)|^{2} d\boldsymbol{x}
+ \frac{\tau\chi}{\gamma}\left\|\frac{c^{n+1}-c^n}{\delta t}\right\|^2_{L^2}\\
&+ (2+\tau) \chi \gamma C_{gn} M^2 \Big),
\end{aligned}
\end{equation}
with the condition of $(2+\tau) \gamma \chi C_{gn} M< 1$.

Let us observe that the first term in the right hand side
of \eqref{3.20-1} is not exactly the same as in the energy dissipation \eqref{e8}.
We overcome this as follows
\begin{equation}\label{3.20-2}
\begin{aligned}
&\int_{\Omega} \rho^{n+1}  |\nabla \left( \log \rho^{n+1}-\chi c^{n+1}  \right)|^{2} d\boldsymbol{x}\\
\leq &2\int_{\Omega} \rho^{n+1}  |\nabla \left( \log \rho^{n+1}-\chi c^{n}  \right)|^{2} d\boldsymbol{x} +2 \int_{\Omega} \chi \rho^{n+1}  |\nabla c^{n+1}- \nabla c^{n}|^{2} d\boldsymbol{x}.
\end{aligned}
\end{equation}
By Cauchy-Schwarz inequality and Gagliardo-Niremberg inequality, it yields
\begin{equation}\label{3.20-4}
\begin{aligned}
&2\int_{\Omega} \chi \rho^{n+1}  |\nabla c^{n+1}- \nabla c^{n}|^{2} \\
\leq& 2\chi \left\|\rho^{n+1}||_{L^2}||\nabla c^{n+1}- \nabla c^{n}\right\|^2_{L^4}\\
\leq& 2 \chi C_{gn}\left\|\rho^{n+1}\right\|_{L^2}
\left\|\nabla c^{n+1}- \nabla c^{n}\right\|_{L^2}
\left\|\Delta c^{n+1}- \Delta c^{n}\right\|_{L^2}\\
\leq&  2 \chi C_{gn} \left\|\rho^{n+1}\right\|_{L^2}
\left\|\nabla c^{n+1}- \nabla c^{n}\right\|_{L^2}
\left( \left\|\Delta c^{n+1}\right\|_{L^2} + \left\|\Delta c^{n}\right\|_{L^2} \right).
\end{aligned}
\end{equation}

Define an operator $A=-\Delta+\alpha I$.
Squaring and integrating \eqref{S2}, we have
\begin{equation}\label{3.20-5}
\left\|A c^{n+1}\right\|_{L^2}^2 +\frac{\tau}{\delta t} \left\|A^{\frac12}c^{n+1}\right\|^2_{L^2}
\leq \frac{\tau}{\delta t} \left\|A^{\frac12}c^{n}\right\|^2_{L^2}+ \gamma^2  \left\|\rho^{n+1}\right\|_{L^2}^2.
\end{equation}
Summing this inequality leads to $\frac{\tau}{\delta t}\left\|A^{\frac12}c^{n}\right\|^2_{L^2}\leq  \gamma^2 \sum\limits_{k=0}^{n} \left\|\rho^{k}\right\|_{L^2}^2$ with $c^{-1}=0$.
Gathering this with \eqref{3.20-5}, we obtain
\begin{equation}\label{lapc1}
\left\|A c^{n+1} \right\|^2_{L^2}
\leq \gamma^2 \sum_{k=0}^{n+1} \left\|\rho^{k}\right\|_{L^2}^2.
\end{equation}

According to the operator theory and functional analysis in \cite{Brezis},
the norm of the linear bounded operator $(-\Delta +\alpha I)^{-1}\Delta$ in $L^2$ is $1$.
To check this, we observe that the symbol of this linear operator is $\frac{ |\xi|^2}{ |\xi|^2+\alpha}$ that is less than $1$.
Then we have
\begin{equation}\label{lapc}
\left\|-\Delta c^{n+1} \right\|^2_{L^2}
\leq \left\|A c^{n+1} \right\|^2_{L^2}
\leq \gamma^2 \sum_{k=0}^{n+1} \left\|\rho^{k}\right\|_{L^2}^2.
\end{equation}

Finally, we substitute this estimate into the inequality \eqref{3.20-4} to get
\begin{equation}\label{3.20-6m}
\begin{aligned}
&2\int_{\Omega} \chi \rho^{n+1}  |\nabla c^{n+1}- \nabla c^{n}|^{2}\\
\leq&  4 \chi C_{gn}\left\|\rho^{n+1}\right\|_{L^2} \left\|\nabla c^{n+1}- \nabla c^{n}\right\|_{L^2}
\left( \sum_{k=0}^{n+1} \gamma^2 \left\|\rho^{k}\right\|_{L^2}^2 \right)^\frac12  \\
\leq& 4 \chi\gamma C_{gn}\|\rho^{n+1}\|_{L^2}^2\left\|\nabla c^{n+1}- \nabla c^{n}\right\|_{L^2}
+4 \chi C_{gn}\left\|\rho^{n+1}\right\|_{L^2} \left\|\nabla c^{n+1}- \nabla c^{n}\right\|_{L^2}
\left( \sum_{k=0}^{n} \gamma^2 \left\|\rho^{k}\right\|_{L^2}^2 \right)^\frac12  \\
\leq& \left(\frac14+4 \chi\gamma C_{gn}\left\|\nabla c^{n+1}- \nabla c^{n}\right\|_{L^2} \right) \left\|\rho^{n+1}\right\|_{L^2}^2
+ 16 \chi^2 \gamma^2 C_{gn}^2 \left\|\nabla c^{n+1}- \nabla c^{n}\right\|_{L^2}^2 \left( \sum_{k=0}^{n} \left\|\rho^{k}\right\|_{L^2}^2 \right).
\end{aligned}
\end{equation}
We derive from  the discrete energy dissipation \eqref{e8}  that
$$4 \chi \gamma C_{gn}\left\|\nabla c^{n+1}- \nabla c^{n}\right\|_{L^2} \leq \left(E_{tot}(\rho_0,c_0) - E_{tot}(\rho^{N_t},c^{N_t})\right)^\frac12 \sqrt{2 \chi} \gamma^\frac32. $$
Assume that $\chi\gamma$ is small enough such that $\left(E_{tot}(\rho_0,c_0) - E_{tot}(\rho^{N_t},c^{N_t})\right)^\frac12 \sqrt{2 \chi} \gamma^\frac32\leq \frac14$,
we gather \eqref{3.20-1}, \eqref{3.20-2} and \eqref{3.20-6m} to obtain
\begin{equation}\label{3.20-7m}
\begin{aligned}
\left\| \rho^{n+1} \right\|^2_{L^2}
\leq  & \widetilde{C}  \left\|\nabla c^{n+1}-\nabla c^n\right\|^2_{L^2}
\left( \sum_{k=0}^n \left\|\rho^{k}\right\|_{L^2}^2 \right)\\
&+ C \Big(4 \int_{\Omega} \rho^{n+1}  |\nabla \left( \log \rho^{n+1}-\chi  c^{n}   \right)|^{2} d\boldsymbol{x} \\
&+ \frac{2\tau\chi}{\gamma}\left\|\frac{c^{n+1}-c^n}{\delta t}\right\|^2_{L^2}
+ (4+2\tau) \chi \gamma C_{gn} M^2 \Big),
\end{aligned}
\end{equation}
where $\widetilde{C}= 16 C\chi^2 \gamma^2 C_{gn}^2$.
Setting $S_{n+1}=\delta t \sum\limits_{k=0}^{n+1} \left\| \rho^{k} \right\|^2_{L^2}$,
then \eqref{3.20-7m} reads
\begin{equation}\label{3.20-8m}
\begin{aligned}
S_{n+1}-S_{n}
\leq & \widetilde{C} \left\|\nabla (c^{n+1}-c^n)\right\|^2_{L^2} S_n \\
&+C \delta t \Big(4 \int_{\Omega} \rho^{n+1}  |\nabla \left( \log \rho^{n+1}-\chi  c^{n}   \right)|^{2} d\boldsymbol{x} \\
&+ \frac{2\tau\chi}{\gamma}\left\|\frac{c^{n+1}-c^n}{\delta t}\right\|^2_{L^2}
+ (4+2\tau) \chi \gamma C_{gn} M^2 \Big).
\end{aligned}
\end{equation}
According to the discrete energy dissipation \eqref{e8}, we have
\begin{equation}\label{hardelot}
\begin{aligned}
&\delta t \sum_{n=0}^{N_t-1} (4 C \int_{\Omega} \rho^{n+1}  |\nabla \left( \log \rho^{n+1}-\chi  c^{n}   \right)|^{2} d\boldsymbol{x}
+ \frac{2\tau\chi}{\gamma}\|\frac{c^{n+1}-c^n}{\delta t}\|^2_{L^2}
+ (4+2\tau) \chi \gamma C_{gn} M^2)\\
\leq & C \left(E_{tot}(\rho_0,c_0) - E_{tot}(\rho^{N_t},c^{N_t}) +  T \right) <  +\infty.
\end{aligned}
\end{equation}
Besides,
\begin{equation}\label{hardelot1}
\begin{aligned}
\tilde C \sum\limits_{n=0}^{N_t-1} \left\|\nabla (c^{n+1}-c^n)\right\|^2_{L^2}
\leq \widetilde{K}_3 :=\tilde C\frac{\gamma}{\chi} \left(E_{tot}(\rho_0,c_0) - E_{tot}(\rho^{N_t},c^{N_t})\right).
\end{aligned}
\end{equation}

Then we take sum of the inequality \eqref{3.20-8m} on $n$ and apply the discrete Gronwall lemma to obtain
\begin{equation}\label{3.20-9}
S_{N_t} = \delta t \sum\limits_{n=0}^{N_t} \left\|\rho^{n}\right\|^2_{L^2}
\leq K_3 <  +\infty,
\end{equation}
where $K_3:=\exp(\widetilde{K}_3)\left\|\rho_0\right\|^2_{L^2} + C\left( E_{tot}(\rho_0,c_0) - E_{tot}(\rho^{N_t},c^{N_t}) + T \right)$.

\end{proof}

\begin{lemma}\label{destL2}
For $n\in \{0,1,2,\ldots,N_t-1\}$, if the  initial data $\rho_0 \in L^2(\Omega)$ and Lemma \ref{destintegral} holds,  then there exists $C>0$ such that
$\|\rho^{n+1}\|_{L^2}\le C$ for the semi-discrete scheme \eqref{S1}-\eqref{S2}.
\end{lemma}
\begin{proof}
Taking the scalar product of \eqref{S1} with $\rho^{n+1}$, this leads to
\begin{equation}\label{3.31}
\begin{aligned}
\left( \frac{\rho^{n+1}-\rho^{n}}{\delta t} ,\rho^{n+1}\right)
-\left( \Delta \rho^{n+1} , \rho^{n+1}\right)
= - \left( \chi\nabla \cdot \left(\rho^{n+1} \nabla c^{n} \right) ,\rho^{n+1}\right)
\end{aligned}
\end{equation}
where
\begin{equation}\label{3.32}
\begin{aligned}
\left( \frac{\rho^{n+1}-\rho^{n}}{\delta t} ,\rho^{n+1}\right)
= \frac{1}{2 \delta t}
\left( \| \rho^{n+1} \|^2_{L^2} - \| \rho^{n} \|^2_{L^2} + \| \rho^{n+1}- \rho^{n}\|^2_{L^2}
\right),
\end{aligned}
\end{equation}
\begin{equation}\label{3.33}
\begin{aligned}
-\left( \Delta \rho^{n+1}  ,\rho^{n+1}\right)
= \| \nabla\rho^{n+1} \|^2_{L^2},
\end{aligned}
\end{equation}
and
\begin{equation}\label{3.34}
\begin{aligned}
- \left( \chi\nabla \cdot \left(\rho^{n+1} \nabla c^{n} \right) ,\rho^{n+1}\right)
= \chi \left( \rho^{n+1} \nabla c^{n}, \nabla \rho^{n+1} \right)
= \frac{\chi}{2} \left( |\rho^{n+1}|^2,  - \Delta c^{n} \right).
\end{aligned}
\end{equation}
Since $c^{n}\geq 0$ for all $n$, we derive from  the equation  \eqref{S2} that
\begin{equation}\label{3.34*}
\begin{aligned}
\frac{\chi}{2} \left( |\rho^{n+1}|^2,  - \Delta c^{n} \right)
&= \frac{\chi}{2} \left( |\rho^{n+1}|^2,  -\tau\frac{c^{n}-c^{n-1}}{\delta t}- \alpha c^{n} + \gamma \rho^{n} \right)  \\
&\leq \frac{\chi \gamma}{2} \left( |\rho^{n+1}|^2,  \rho^{n} \right)
-\frac{\tau\chi}{2} \left( |\rho^{n+1}|^2, \frac{ c^n-c^{n-1}}{\delta t} \right).
\end{aligned}
\end{equation}
Applying Cauchy-Schwarz inequality, Gagliardo-Nirenberg inequality and Young's inequality in turn, we derive
\begin{equation}\label{3.35}
\begin{aligned}
\frac{\chi \gamma}{2} \left( |\rho^{n+1}|^2, \rho^n \right)
&\leq \frac{\chi \gamma}{2} \|\rho^n\|_{L^2} \|(\rho^{n+1})^2\|_{L^2} \\
&\leq  \frac{\chi \gamma C_{gn}}{2} \|\rho^n\|_{L^2} \|\rho^{n+1}\|_{L^2} (\|\rho^{n+1}\|_{L^2}+\|\nabla\rho^{n+1}\|_{L^2})  \\
&\leq \frac{\chi \gamma C_{gn}}{2}
\left(\frac 1 \varepsilon\|\rho^n\|_{L^2}^2+\|\rho^n\|_{L^2} \right) \|\rho^{n+1}\|^2_{L^2}
+ \frac{\chi \gamma C_{gn}}{2} \varepsilon  \|\nabla\rho^{n+1}\|^2_{L^2},
\end{aligned}
\end{equation}
and similarly

\begin{equation}\label{3.35add}
\begin{aligned}
-\frac{\tau\chi}{2} \left( |\rho^{n+1}|^2, \frac{ c^n-c^{n-1}}{\delta t} \right)
\leq& \frac{\tau\chi C_{gn}}{2}
\left(\frac{1}{\varepsilon } \left\| \frac{ c^n- c^{n-1}}{\delta t} \right\|_{L^2}^2 + \left\| \frac{c^n- c^{n-1}}{\delta t} \right\|_{L^2} \right) \|\rho^{n+1}\|^2_{L^2}\\
+& \frac{\tau\chi C_{gn}}{2} \varepsilon \|\nabla\rho^{n+1}\|^2_{L^2}.
\end{aligned}
\end{equation}
Substituting \eqref{3.32}-\eqref{3.35add} into \eqref{3.31}, we obtain
\begin{equation}\label{3.36}
\begin{aligned}
&\frac{1}{2 \delta t}
\left( \| \rho^{n+1} \|^2_{L^2} - \| \rho^{n} \|^2_{L^2} + \| \rho^{n+1}- \rho^{n}\|^2_{L^2}
\right) + \| \nabla\rho^{n+1} \|^2_{L^2}  \\
\leq &
  \frac{\chi C_{gn}}{2} \left(\frac \gamma \varepsilon \|\rho^n\|_{L^2}^2+\gamma\|\rho^n\|_{L^2}
+ \frac{\tau}{\varepsilon } \left\| \frac{ c^n- c^{n-1}}{\delta t} \right\|_{L^2}^2
+ \tau \left\| \frac{c^n- c^{n-1}}{\delta t} \right\|_{L^2}\right)
\|\rho^{n+1}\|^2_{L^2}  \\
+ &\frac{\chi C_{gn}}{2}\varepsilon\left(\gamma +\tau\right)
\|\nabla\rho^{n+1}\|^2_{L^2}.
\end{aligned}
\end{equation}

As long as $\varepsilon$ is small enough so that $\frac{\chi C_{gn}}{2} \left(\gamma +\tau\right) \varepsilon\leq 1$, we have
\begin{equation}\label{3.37}
\begin{aligned}
\| \rho^{n+1} \|^2_{L^2}
\leq &
\| \rho^{n} \|^2_{L^2}
+\chi C_{gn}\delta t \Big(\frac \gamma \varepsilon \|\rho^n\|_{L^2}^2+\gamma\|\rho^n\|_{L^2}
+ \frac{\tau}{\varepsilon } \left\| \frac{ c^n- c^{n-1}}{\delta t} \right\|_{L^2}^2 \\
+ & \tau \left\| \frac{c^n- c^{n-1}}{\delta t} \right\|_{L^2}\Big)
\|\rho^{n+1}\|^2_{L^2}.
\end{aligned}
\end{equation}
Since $\delta t \sum_{n=0}^{N_t} \|\rho^{n}\|^2_{L^2} \leq K_3$ in Lemma \ref{destintegral} and $\sum_{n=1}^{N_t} \| \frac{ c^n- c^{n-1}}{\delta t} \|_{L^2}^2$ is bounded due to the discrete energy dissipation formula \eqref{e8},
we take the sum of \eqref{3.37} on $n$ and apply the discrete Gronwall lemma to obtain
\begin{equation}\label{3.38}
\| \rho^{n+1} \|^2_{L^2}
\leq K_4< +\infty,
\end{equation}
where
\begin{equation}\label{3.39}
\begin{aligned}
K_4:=\| \rho_0 \|^2_{L^2} \exp \left(CT \left(
K_3+K_3^{\frac{1}{2}} + \tau \left\| \frac{ c^n- c^{n-1}}{\delta t} \right\|_{L^2}^2
+ \tau \left\| \frac{c^n- c^{n-1}}{\delta t} \right\|_{L^2} \right)\right),
\end{aligned}
\end{equation}
with $\frac 1 \varepsilon$ being absorbed by the constant $C$.

\end{proof}

\begin{lemma}\label{destLp}
For $n\in \{0,1,2,\ldots,N_t-1\}$ and $1<p<+\infty$, if the initial data $\rho_0 \in L^p(\Omega)$ and Lemma \ref{destL2} holds,  then there exists $C>0$ such that  $\|\rho^{n+1}\|_{L^p}\le C$ for the semi-discrete scheme \eqref{S1}-\eqref{S2}.
\end{lemma}
\begin{proof}
Taking the scalar product of \eqref{S1} with $(\rho^{n+1})^{2p-1}$, this leads to
\begin{align}\label{nd1}
\left( \frac{\rho^{n+1}-\rho^{n}}{\delta t} ,(\rho^{n+1})^{2p-1}\right)
-\left( \Delta \rho^{n+1} , (\rho^{n+1})^{2p-1}\right)
= - \left( \chi\nabla \cdot \left(\rho^{n+1} \nabla c^{n} \right) ,(\rho^{n+1})^{2p-1}\right)
\end{align}
where
\begin{equation}\label{nd2}
\begin{aligned}
\frac{1}{\delta t} \left( \rho^{n}, (\rho^{n+1})^{2p-1} \right)
\leq &\frac{1}{\delta t} \|\rho^{n}\|_{L^{2p}}
\|\rho^{n+1}\|_{L^{2p}}^{2p-1}\\
\leq &\frac{1}{\delta t} \left( \frac{1}{2p} \|\rho^{n}\|_{L^{2p}}^{2p}
+\frac{2p-1}{2p} \|\rho^{n+1}\|_{L^{2p}}^{2p} \right),
\end{aligned}
\end{equation}
\begin{align}\label{nd2*}
-\left( \Delta \rho^{n+1}  ,(\rho^{n+1})^{2p-1}\right)
= \frac{2p-1}{p^2}\| \nabla((\rho^{n+1})^p) \|^2_{L^2},
\end{align}
and
\begin{equation}\label{nd3}
\begin{aligned}
- \left( \chi\nabla \cdot \left(\rho^{n+1} \nabla c^{n} \right) ,(\rho^{n+1})^{2p-1}\right)
=& \chi (2p-1) \left( (\rho^{n+1})^{2p-1} \nabla c^{n}, \nabla \rho^{n+1} \right)\\
=& \frac{(2p-1)\chi}{2p} \left( |\rho^{n+1}|^{2p},  - \Delta c^{n} \right).
\end{aligned}
\end{equation}

According to the equation about $c^{n}$ \eqref{S2} and $c^{n}\geq 0$ for all $n$, we have
\begin{equation}\label{nd4}
\begin{aligned}
\left( |\rho^{n+1}|^{2p},  - \Delta c^{n} \right)
&=\left( |\rho^{n+1}|^{2p}, -\tau\frac{c^{n}-c^{n-1}}{\delta t}- \alpha c^{n} + \gamma \rho^{n} \right) \\
&\leq \gamma \left( |\rho^{n+1}|^{2p}, \rho^{n} \right)
-\tau\left( |\rho^{n+1}|^{2p}, \frac{c^n-c^{n-1}}{\delta t} \right).
\end{aligned}
\end{equation}

Applying Cauchy-Schwarz inequality, Gagliardo-Nirenberg inequality and Young's inequality in turn, we derive
\begin{equation}\label{nd4-1}
\begin{aligned}
\gamma \left( |\rho^{n+1}|^{2p}, \rho^n \right)
&\leq \gamma \|\rho^n\|_{L^2} \|(\rho^{n+1})^{2p}\|_{L^2} \\
&\leq \gamma C_{gn} \|\rho^n\|_{L^2} \|(\rho^{n+1})^p\|_{L^2}
\left(\|\nabla((\rho^{n+1})^p)\|_{L^2} + \|(\rho^{n+1})^p\|_{L^2}\right) \\
&\leq \gamma C_{gn} \left(\frac 1 \varepsilon \|\rho^n\|_{L^2}^2 + \|\rho^n\|_{L^2} \right)
\|(\rho^{n+1})^p\|_{L^2}^2
+ \gamma C_{gn} \varepsilon \|\nabla((\rho^{n+1})^p)\|_{L^2}^2 \\
&\leq \gamma C_{gn} \left( \frac{1}{\varepsilon} K_4+ K_4^{\frac{1}{2}}  \right)
\|(\rho^{n+1})^p\|_{L^2}^2
+ \gamma C_{gn} \varepsilon  \|\nabla((\rho^{n+1})^p)\|_{L^2}^2,
\end{aligned}
\end{equation}
where $\|\rho^{n}\|^2_{L^2} \leq K_4$ is known in Lemma \ref{destL2}.
Similarly,
\begin{equation}\label{nd4-2}
\begin{aligned}
 -\tau\left( |\rho^{n+1}|^{2p}, \frac{c^n-c^{n-1}}{\delta t} \right)
\leq &\tau C_{gn} \left(\frac{1}{\varepsilon} \left\|\frac{c^n-c^{n-1}}{\delta t} \right\|^2_{L^2}
+\left\|\frac{c^n-c^{n-1}}{\delta t} \right\|_{L^2} \right)
\|(\rho^{n+1})^p\|_{L^2}^2\\
+ &\tau C_{gn} \varepsilon  \|\nabla((\rho^{n+1})^p)\|_{L^2}^2.
\end{aligned}
\end{equation}
Gathering these inequalities above, we obtain
\begin{equation}\label{nd5}
\begin{aligned}
&\frac{1}{\delta t} \|\rho^{n+1}\|_{L^{2p}}^{2p}
+ \frac{2p-1}{p^2}\| \nabla((\rho^{n+1})^p) \|^2_{L^2} \\
\leq & \frac{1}{\delta t} \left( \frac{1}{2p} \|\rho^{n}\|_{L^{2p}}^{2p}
+\frac{2p-1}{2p} \|\rho^{n+1}\|_{L^{2p}}^{2p} \right)
+ \frac{(2p-1)\chi C_{gn} \varepsilon}{2p} \left( \gamma + \tau \right)
\|\nabla((\rho^{n+1})^p)\|_{L^2}^2 \\
+ & \frac{(2p-1)\chi C_{gn} }{2p} \left( \frac{\gamma}{\varepsilon} K_4 + \gamma K_4^{\frac{1}{2}}
+ \frac{\tau}{\varepsilon} \left\|\frac{c^n-c^{n-1}}{\delta t} \right\|^2_{L^2} +  \tau \left\|\frac{c^n-c^{n-1}}{\delta t} \right\|_{L^2} \right)
\|(\rho^{n+1})^p\|_{L^2}^2.
\end{aligned}
\end{equation}
If $\varepsilon$ is small enough so that $\frac{\chi p C_{gn}}{2} (\gamma + \tau) \varepsilon\leq 1$, we have
\begin{equation}\label{nd6}
\begin{aligned}
\|\rho^{n+1}\|_{L^{2p}}^{2p}
\leq &\|\rho^{n}\|_{L^{2p}}^{2p}
+ (2p-1)\chi C_{gn} \delta t \Big( \frac{\gamma}{\varepsilon} K_4 + \gamma K_4^{\frac{1}{2}}
+ \frac{\tau}{\varepsilon} \left\|\frac{c^n-c^{n-1}}{\delta t} \right\|^2_{L^2} \\
+&  \tau \left\|\frac{c^n-c^{n-1}}{\delta t} \right\|_{L^2} \Big)
\|(\rho^{n+1})^p\|_{L^2}^2.
\end{aligned}
\end{equation}
Since $\sum_{n=1}^{N_t} \left\| \frac{ c^n- c^{n-1}}{\delta t} \right\|_{L^2}^2$ is bounded due to the discrete energy dissipation \eqref{e8},
we take the sum of \eqref{nd6} on $n$ and apply the discrete Gronwall lemma to obtain
\begin{equation}\label{nd7}
\|\rho^{n+1}\|_{L^{2p}}^{2p}
\leq K_5< +\infty, \ \ 1<p<\infty,
\end{equation}
where
\begin{equation}\label{nd8}
\begin{aligned}
K_5:=\| \rho_0 \|_{L^{2p}}^{2p}
\exp \left( CT \left(
K_4+K_4^{\frac{1}{2}} + \tau \left\| \frac{ c^n- c^{n-1}}{\delta t} \right\|_{L^2}^2
+ \tau \left\| \frac{c^n- c^{n-1}}{\delta t} \right\|_{L^2} \right)\right),
\end{aligned}
\end{equation}
with $\frac 1 \varepsilon$ being absorbed by the constant $C$.

\end{proof}

\section{Error estimates}
Now we rigorously derive the error estimate in $L^p$-norm $(1<p<\infty)$ for the semi-implicit Euler scheme \eqref{S1}-\eqref{S2}.
Since we have to compute consistency errors, we assume here that the solutions
$\rho$ and $c$ are as smooth as needed. We recall from Section \ref{bounds} that $\rho$ belongs to $L^\infty(\Omega)$.

Firstly we define error functions as follows:
\begin{equation}\label{error}
e_{\rho}^{n+1} = \rho(\cdot,t^{n+1}) - \rho^{n+1}, \ \
e_c^{n+1} = c(\cdot,t^{n+1}) - c^{n+1}, \ \ n=0,1,2,\ldots,N_t-1.
\end{equation}
Then we replace approximate solutions $(\rho^{n}, c^{n})$ in semi-discrete scheme \eqref{S1}-\eqref{S2} with exact solutions $(\rho(\cdot,t^{n}), c(\cdot,t^{n}))$ to get
\begin{align}
&\frac{\rho(t^{n+1})-\rho(t^{n})}{\delta t}
- \Delta \rho(t^{n+1}) + \chi\nabla \cdot \left(\rho(t^{n+1}) \nabla c(t^{n}) \right)
= T_{\rho}(t^{n+1}),\label{C1} \\
&
\tau\frac{c(t^{n+1})-c(t^{n})}{\delta t}- \Delta c(t^{n+1}) + \alpha c(t^{n+1}) - \gamma \rho(t^{n+1})
= \tau T_{c}(t^{n+1}),\label{C2}
\end{align}
Assume that the solutions $(\rho, c)$ of the KS equations \eqref{2.1}-\eqref{2.2} are smooth enough, by a standard Taylor expansion at $t^{n+1}$, we have the truncation error
\begin{align}
&T_{\rho}(t^{n+1}) = - \delta t \left(
\frac{1}{2} \frac{\partial^2 \rho(t^{n+1})}{\partial t^2}
+ \chi\nabla \cdot \left(\rho(t^{n+1}) \nabla \frac{\partial c(t^{n+1})}{\partial t}\right) \right), \label{Trho} \\
&T_{c}(t^{n+1}) = - \delta t \left(
\frac{1}{2} \frac{\partial^2 c(t^{n+1})}{\partial t^2} \right).  \label{Tc}
\end{align}

Subtracting the semi-discrete equation \eqref{S1} from the consistency estimate \eqref{C1}, we have
\begin{equation}\label{eq1}
\begin{aligned}
\frac{e_\rho^{n+1}-e_\rho^{n}}{\delta t}
-\Delta e_\rho^{n+1}
+ \chi\nabla \cdot \left(\rho(t^{n+1}) \nabla c(t^{n}) - \rho^{n+1} \nabla c^{n} \right)
= T_{\rho}(t^{n+1}).
\end{aligned}
\end{equation}
Since that
\begin{equation}\label{eq2}
\begin{aligned}
&\rho(t^{n+1}) \nabla c(t^{n}) - \rho^{n+1} \nabla c^{n}\\
= &\rho(t^{n+1}) \nabla c(t^{n}) - \rho(t^{n+1}) \nabla c^{n} +  \rho(t^{n+1}) \nabla c^{n} - \rho^{n+1} \nabla c^{n}\\
= &\rho(t^{n+1}) \nabla e_c^{n} + e_\rho^{n+1} \nabla c^{n},
\end{aligned}
\end{equation}
we can rewrite the error equation for $e_\rho^{n+1}$ as
\begin{equation}\label{erho}
\begin{aligned}
\frac{e_\rho^{n+1}-e_\rho^{n}}{\delta t}
-\Delta e_\rho^{n+1}
+ \chi\nabla \cdot \left( \rho(t^{n+1}) \nabla e_c^{n} + e_\rho^{n+1} \nabla c^{n} \right)
= T_{\rho}(t^{n+1}).
\end{aligned}
\end{equation}
Subtract the semi-discrete equation \eqref{S2} from the consistency estimate \eqref{C2}, we obtain the error equation for $e_c^{n+1}$

\begin{equation}\label{ec}
\begin{aligned}
\tau\frac{e_c^{n+1}-e_c^{n}}{\delta t}
-\Delta e_c^{n+1} + \alpha e_c^{n+1} - \gamma e_\rho^{n+1}
= \tau T_{c}(t^{n+1}).
\end{aligned}
\end{equation}

When analyzing the $L^p$-bound of $e_\rho^{n+1}$, the estimation of $\nabla e_c^{n}$ is a key point, so we give the following lemma in advance.

\begin{lemma}\label{gradientC}
For the parabolic-elliptic case $\tau=0$, there exists  $C>0$ such that
\begin{equation}\label{rg3+}
\| \nabla e_c^{n}\|_{L^{2p}}^{2p} \leq C \|e_\rho^{n+1}\|_{L^{2p}}^{2p}.
\end{equation}
For the parabolic-parabolic case $\tau>0$, there exists  $C>0$ such that
\begin{equation}\label{fin}
\max_{k\leq n} \| \nabla e_c^{k}\|_{L^{2p}}^{2p} \leq C \left(\max_{k\leq n} \|e_\rho^{k}\|_{L^{2p}}^{2p}+\delta t \right).
\end{equation}
\end{lemma}

\begin{proof}
For the parabolic-elliptic case $\tau=0$, the error equation \eqref{ec} for $e_c^{n+1}$ reads
\begin{equation}\label{rg2+}
-\Delta e_c^{n+1} + \alpha e_c^{n+1} =\gamma e_\rho^{n+1}.
\end{equation}
Then \eqref{rg3+} follows from the standard elliptic estimate.

For the parabolic-parabolic case $\tau>0$, the error equation \eqref{ec} for $e_c^{n+1}$ reads
\begin{equation}\label{fin1}
(I+\frac{\delta t}{\tau} A)e_c^{n+1}=e_c^{n}+ \delta t T_{c}(t^{n+1})+ \frac{\gamma\delta t}{\tau}e_\rho^{n+1},
\end{equation}
where we have set the operator $A=-\Delta+\alpha I$.
With $e_c^0=0$, we have
\begin{equation}\label{fin2}
\nabla e_c^{n}=\delta t \sum_{k=0}^{n} \nabla (I+\frac{\delta t}{\tau} A)^{k-n-1} \left(T_{c}(t^{k})+\frac{\gamma}{\tau} e_\rho^{k} \right).
\end{equation}
We recall the Sobolev embedding $H^s(\Omega) \subset L^{2p}(\Omega)$ for $s=1-\frac1p$ in \cite{Brezis} to get
\begin{equation}\label{jo1}
\left\|\nabla e_c^{n} \right\|_{L^{2p}}
\leq C \left\| A^{\frac s 2} \nabla e_c^n \right\|_{L^2}
\leq C \delta t \sum_{k=0}^{n} \left\|\nabla  A^{\frac s 2} (I+\frac{\delta t}{\tau} A)^{k-n-1} \right\|_{\mathcal{L}(L^2,L^2)}
\left\|T_{c}(t^{k})+\frac{\gamma}{\tau} e_\rho^{k} \right\|_{L^2}.
\end{equation}
Appealing Plancherel's theorem \cite{PDE} we have that, for $k\geq 1$

\begin{equation}\label{jo2}
\left\|\nabla A^{\frac s 2} (I+\frac{\delta t}{\tau} A)^{-k} \right\|_{\mathcal{L}(L^2,L^2)}
=\sup_{|\xi|}  \left|\frac{|\xi|(\alpha + |\xi|^2)^{\frac s 2}}{\left(1+\frac{\delta t}{\tau}(\alpha+|\xi|^2)\right)^{k}} \right|  .
\end{equation}

\noindent To bound by above the right hand side of \eqref{jo2} we consider two cases, namely
$|\xi|^2\leq \alpha$ and $\alpha<|\xi|^2$. In the former case the right hand side
is bounded by $2^{\frac s 2}\alpha^{\frac 1 2 +\frac s 2}$. In the later case
we observe that the maximum of the function $|\xi|\mapsto \frac{ |\xi|^{1+s}}{(1+\frac{\delta t}{\tau} |\xi|^2)^k}$
is achieved for $\frac{\delta t}{\tau} |\xi|^2(2k-1-s)=1+s. $
Then the right hand side of \eqref{jo2} is bounded by above by
$\max\{2^{\frac s 2}\alpha^{\frac 1 2 +\frac s 2}, (\frac{(1+s)\tau}{(2k-1-s)\delta t}){^\frac{1+s}{2}}\}$.
Since
\begin{equation}\label{jo3}
\sum_{k=1}^n (2k-1-s)^{-\frac{1+s}{2}}\leq C_s n^{\frac12-\frac s 2}\leq C_s (\delta t)^{\frac s 2 -\frac 1 2} T^{\frac12-\frac s 2},
\end{equation}
we have
\begin{equation}\label{jo4}
\delta t \sum_{k=0}^n \left\|\nabla  A^{\frac s 2} (I+\frac{\delta t}{\tau} A)^{k-n-1} \right\|_{\mathcal{L}(L^2,L^2)}\leq C,
\end{equation}
and the conclusion follows promptly, appealing the consistency estimate \eqref{Tc} and the embedding $L^{2p}(\Omega) \subset L^2(\Omega)$.
\end{proof}

With the above preparations, we re ready to prove the following error estimates.
\begin{theorem}
Assume that exact solutions $\rho$ and $c$ of the general KS model \eqref{2.1}-\eqref{1.4} are smooth enough for a fixed final time $T \in (0,T_{\text{max}})$,
the positive initial data $\rho_0, c_0 \in L^q(\Omega)$ $(1<q\leq\infty)$, and the mass $M=\int_{\Omega} \rho_0 d\boldsymbol{x}$ is small enough to ensure $(2+\tau)\gamma \chi C_{gn}M< 1$,
where $T_{ \text{max}}$ is the maximal existence time of solutions and
$C_{gn}$ is the positive constant coefficient of Gagliardo-Nirenberg inequality in 2D  domain $\Omega$, then for the semi-discrete scheme \eqref{S1}-\eqref{S2} we have the estimate
\begin{align}\label{th2}
\|e_\rho^{n+1}\|_{L^p} + \|e_c^{n+1}\|_{L^p}
\leq C\delta t, \ \ n=0,1,2,\ldots,N_t-1, \ \ 1<p<\infty.
\end{align}
\end{theorem}

\begin{comment}
Before we set out to prove the above theorem, we make some remarks on the assumptions of this theorem.
Assuming that exact solutions $\rho$ and $c$ are smooth enough allows us to have consistency errors.
We suppose that $\rho_0, c_0$ are positive to ensure that the exact solution and the discrete solution remain also positive.
Our result also requires the smallness of the mass $M=\int_{\Omega} \rho_0 d\boldsymbol{x}$
to use Lemma \ref{estintegral} and Lemma \ref{destintegral}.
\end{comment}

\begin{proof}
Taking the inner product of the error equation \eqref{erho} with $(e_\rho^{n+1})^{2p-1}$, this leads to
\begin{equation}\label{erhop1}
\begin{aligned}
&\left( \frac{e_\rho^{n+1}-e_\rho^{n}}{\delta t}, (e_\rho^{n+1})^{2p-1} \right)
- \left( \triangle e_\rho^{n+1},(e_\rho^{n+1})^{2p-1}\right)\\
=& -\chi \left( \nabla \cdot \left( \rho(t^{n+1}) \nabla e_c^{n} + e_\rho^{n+1} \nabla c^{n}\right), (e_\rho^{n+1})^{2p-1} \right)
+ \left( T_{\rho}(t^{n+1}),  (e_\rho^{n+1})^{2p-1}\right).
\end{aligned}
\end{equation}
Next, we estimate the bound of each term of this equation.

For the first term, we have
\begin{equation}\label{erhop2}
\begin{aligned}
&\left( \frac{e_\rho^{n+1}-e_\rho^{n}}{\delta t}, (e_\rho^{n+1})^{2p-1} \right)
= \frac{1}{\delta t} \|e_\rho^{n+1}\|_{L^{2p}}^{2p}
- \frac{1}{\delta t} \left( e_\rho^{n}, (e_\rho^{n+1})^{2p-1} \right),
\end{aligned}
\end{equation}
and
\begin{equation}\label{erhop3}
\begin{aligned}
\frac{1}{\delta t} \left( e_\rho^{n}, (e_\rho^{n+1})^{2p-1} \right)
\leq &\frac{1}{\delta t} \|e_\rho^{n}\|_{L^{2p}}
\|e_\rho^{n+1}\|_{L^{2p}}^{2p-1}\\
\leq &\frac{1}{\delta t} \left( \frac{1}{2p} \|e_\rho^{n}\|_{L^{2p}}^{2p}
+\frac{2p-1}{2p} \|e_\rho^{n+1}\|_{L^{2p}}^{2p} \right)
\end{aligned}
\end{equation}
by H\"{o}lder inequality and Young's inequality.

For the second term,
\begin{equation}\label{erhop4}
\begin{aligned}
- \left( \triangle e_\rho^{n+1},(e_\rho^{n+1})^{2p-1}\right)
=\frac{2p-1}{p^2} \| \nabla (e_\rho^{n+1})^p \|_{L^2}^2.
\end{aligned}
\end{equation}

For the last term, using H\"{o}lder inequality and Young's inequality, we obtain
\begin{equation}\label{erhop5}
\begin{aligned}
\left( T_{\rho}(t^{n+1}),  (e_\rho^{n+1})^{2p-1}\right)
\leq & \| T_\rho(t^{n+1}) \|_{L^{2p}}  \| e_\rho^{n+1} \|_{L^{2p}}^{2p-1} \\
\leq & \frac{1}{2p} \| T_\rho(t^{n+1}) \|_{L^{2p}}^{2p} + \frac{2p-1}{2p} \| e_\rho^{n+1} \|_{L^{2p}}^{2p}.
\end{aligned}
\end{equation}

For the the third and the most complex term, we treat it as two parts $-\chi \left( \nabla \cdot ( e_\rho^{n+1} \nabla c^{n}\right),$ $(e_\rho^{n+1})^{2p-1} )$ and $-\chi \left( \nabla \cdot \left( \rho(t^{n+1}) \nabla e_c^{n} \right), (e_\rho^{n+1})^{2p-1} \right)$ below.
According to the Green formula, the equation about $c^{n}$ \eqref{S2} and the positivity of $c^{n}$ for all $n$, the first part is estimated as
\begin{equation}\label{erhop6}
\begin{aligned}
&-\chi \left( \nabla \cdot \left(  e_\rho^{n+1} \nabla c^{n}\right), (e_\rho^{n+1})^{2p-1} \right)\\
=& \chi (2p-1) \left( ( e_\rho^{n+1})^{2p-1} \nabla c^{n}, \nabla  e_\rho^{n+1} \right) \\
=& \frac{(2p-1)\chi}{2p} \left( | e_\rho^{n+1}|^{2p},  - \Delta c^{n} \right)\\
=&
\frac{(2p-1)\chi}{2p} \left( | e_\rho^{n+1}|^{2p},  -\tau\frac{c^{n}-c^{n-1}}{\delta t}- \alpha c^{n} + \gamma \rho^{n}  \right)    \\
\leq& \frac{(2p-1)\chi}{2p} \left( | e_\rho^{n+1}|^{2p},  \gamma \rho^n \right)
 -\frac{(2p-1)\chi}{2p}\left( | e_\rho^{n+1}|^{2p}, \tau\frac{c^{n}-c^{n-1}}{\delta t}  \right).
\end{aligned}
\end{equation}
By Cauchy-Schwarz inequality, Gagliardo-Nirenberg inequality and Young's inequality, we derive
\begin{equation}\label{erhop6-1}
\begin{aligned}
\left( | e_\rho^{n+1}|^{2p}, \gamma \rho^n \right)
\leq& \gamma \|\rho^n\|_{L^2} \|(e_\rho^{n+1})^{2p}\|_{L^2}\\
\leq& \gamma C_{gn}  \|\rho^n\|_{L^2} \|(e_\rho^{n+1})^{p}\|_{L^2}
\left( \|\nabla(e_\rho^{n+1})^{p}\|_{L^2} + \|(e_\rho^{n+1})^{p}\|_{L^2} \right)\\
\leq& \gamma C_{gn} \left(\frac{1}{\varepsilon} \|\rho^n\|_{L^2}^2 + \|\rho^n\|_{L^2} \right)
\|(e_\rho^{n+1})^{p}\|_{L^2}^2
+\gamma C_{gn} \varepsilon  \|\nabla(e_\rho^{n+1})^{p}\|_{L^2}^2\\
\leq& \gamma C_{gn} \left( \frac{1}{\varepsilon}K_4 + K_4^{\frac{1}{2}} \right)
\|(e_\rho^{n+1})^{p}\|_{L^2}^2
+\gamma C_{gn} \varepsilon  \|\nabla(e_\rho^{n+1})^{p}\|_{L^2}^2,
\end{aligned}
\end{equation}
where $\|\rho^n\|^2_{L^2} \leq K_4$ for any $n$ is known in Lemma 5.3.
Similarly,
\begin{equation}\label{erhop6-2}
\begin{aligned}
&-\tau\left( | e_\rho^{n+1}|^{2p}, \frac{c^{n}-c^{n-1}}{\delta t}\right)\\
\leq& \tau  C_{gn} \left( \frac{1}{\varepsilon} \left\|\frac{c^{n}-c^{n-1}}{\delta t} \right\|_{L^2}^2 + \left\|\frac{c^{n}-c^{n-1}}{\delta t} \right\|_{L^2} \right)
\|(e_\rho^{n+1})^{p}\|_{L^2}^2
+ \tau  C_{gn} \varepsilon  \|\nabla(e_\rho^{n+1})^{p}\|_{L^2}^2
\end{aligned}
\end{equation}
Using Cauchy-Schwarz inequality, H\"{o}lder inequality and Young's inequality, besides, applying $L^\infty$-bound of exact solution $\rho$ in Section \ref{bounds}, the second part is estimated as

\begin{equation}\label{erhop7}
\begin{aligned}
&-\chi \left( \nabla \cdot \left( \rho(t^{n+1}) \nabla e_c^{n} \right), (e_\rho^{n+1})^{2p-1} \right)\\
=& \chi \left( \rho(t^{n+1}) \nabla e_c^{n}, \nabla (e_\rho^{n+1})^{2p-1} \right) \\
=& \frac{(2p-1)\chi }{p} \left( \rho(t^{n+1}) \nabla e_c^{n}, (e_\rho^{n+1})^{p-1} \nabla (e_\rho^{n+1})^{p} \right) \\
\leq& \frac{(2p-1)\chi }{p} \| \rho(t^{n+1}) (\nabla e_c^{n}) (e_\rho^{n+1})^{p-1} \|_{L^2}
\| \nabla (e_\rho^{n+1})^{p} \|_{L^2}\\
\leq& \frac{(2p-1)\chi }{p} \| \rho(t^{n+1}) (\nabla e_c^{n})\|_{L^{2p}}
\|e_\rho^{n+1} \|_{L^{2p}}^{p-1}
 \| \nabla (e_\rho^{n+1})^{p} \|_{L^2}\\
\leq&  \frac{(2p-1)\chi }{2p}
\left(\frac 1 \varepsilon \| \rho(t^{n+1}) (\nabla e_c^{n})\|_{L^{2p}} ^2\|e_\rho^{n+1}\|_{L^{2p}}^{2p-2} +\varepsilon \|\nabla(e_\rho^{n+1})^{p}\|_{L^2}^2 \right)\\
\leq&  \frac{(2p-1)\chi }{2p\varepsilon}
\left(\frac{1}{p}\|\rho(t^{n+1}) (\nabla e_c^{n})\|_{L^{2p}}^{2p} + \frac{2p-2}{2p}\|e_\rho^{n+1}\|_{L^{2p}}^{2p} \right)
+ \frac{(2p-1)\chi \varepsilon}{2p} \|\nabla(e_\rho^{n+1})^{p}\|_{L^2}^2 \\
\leq&  \frac{(2p-1)\chi }{2p\varepsilon}
\left(\frac{1}{p}\| \rho(t^{n+1})\|^{2p}_{L^{\infty}} \|\nabla e_c^{n}\|^{2p}_{L^{2p}}
+ \frac{2p-2}{2p}\|e_\rho^{n+1}\|_{L^{2p}}^{2p} \right)
+ \frac{(2p-1)\chi \varepsilon}{2p} \|\nabla(e_\rho^{n+1})^{p}\|_{L^2}^2 \\
\leq&  \frac{(2p-1)\chi }{2p\varepsilon}
\left(\frac{1}{p} C \|\nabla e_c^{n}\|^{2p}_{L^{2p}}
+ \frac{2p-2}{2p}\|e_\rho^{n+1}\|_{L^{2p}}^{2p} \right)
+ \frac{(2p-1)\chi \varepsilon}{2p} \|\nabla(e_\rho^{n+1})^{p}\|_{L^2}^2.
\end{aligned}
\end{equation}

Substituting \eqref{erhop2}-\eqref{erhop7} into \eqref{erhop1}, we have
\begin{equation}\label{erhop9}
\begin{aligned}
& \frac{1}{\delta t} \|e_\rho^{n+1}\|_{L^{2p}}^{2p}
+ \frac{2p-1}{p^2} \| \nabla (e_\rho^{n+1})^p \|_{L^2}^2 \\
\leq &\frac{1}{\delta t} \left( \frac{1}{2p} \|e_\rho^{n}\|_{L^{2p}}^{2p}
+\frac{2p-1}{2p} \|e_\rho^{n+1}\|_{L^{2p}}^{2p} \right)
+\frac{1}{2p} \| T_\rho(t^{n+1}) \|_{L^{2p}}^{2p} + \frac{2p-1}{2p} \| e_\rho^{n+1} \|_{L^{2p}}^{2p} \\
+ & \frac{(2p-1)\chi C_{gn} }{2p}
\left( \frac{\gamma}{\varepsilon} K_4 + \gamma K_4^{\frac{1}{2}} + \frac{\tau}{\varepsilon} \left\|\frac{c^{n}-c^{n-1}}{\delta t} \right\|_{L^2}^2 + \tau \left\|\frac{c^{n}-c^{n-1}}{\delta t} \right\|_{L^2} \right) \| e_\rho^{n+1} \|_{L^{2p}}^{2p}  \\
+ & \frac{(2p-1)\chi \varepsilon }{2p} \left(  C_{gn}(\gamma + \tau)+1 \right) \|\nabla(e_\rho^{n+1})^{p}\|_{L^2}^2 \\
+ &\frac{(2p-1)\chi }{2p\varepsilon}
\left(\frac{C}{p} \| \nabla e_c^{n}\|_{L^{2p}}^{2p} + \frac{2p-2}{2p}\|e_\rho^{n+1}\|_{L^{2p}}^{2p} \right).
\end{aligned}
\end{equation}
We now chose $\varepsilon$ small enough such that $\frac{p\chi \varepsilon }{2} \left(  C_{gn}(\gamma + \tau)+1 \right)\leq 1$.
Appealing the discrete energy dissipation \eqref{e8}, we know that the $\frac{c^{n}-c^{n-1}}{\delta t}$ remains bounded in $L^2$.
Then we infer from \eqref{erhop9} that there exists a constant $C$, such that
\begin{equation}\label{rg1}
\begin{aligned}
 \|e_\rho^{n+1}\|_{L^{2p}}^{2p}
\leq & \|e_\rho^{n}\|_{L^{2p}}^{2p}
+  C \delta t \left(\| T_\rho(t^{n+1}) \|_{L^{2p}}^{2p}+ \|e_\rho^{n+1}\|_{L^{2p}}^{2p}+
C \|\nabla e_c^{n}\|^{2p}_{L^{2p}} \right).
\end{aligned}
\end{equation}
According to the estimation of $\|\nabla e_c^{n}\|^{2p}_{L^{2p}}$ in Lemma \ref{gradientC},
we should discuss the parabolic-elliptic case and the parabolic-parabolic case, respectively.

In the parabolic-elliptic case,
appealing \eqref{rg1} leads to
\begin{equation}\label{rg4+}
\begin{aligned}
 \|e_\rho^{n+1}\|_{L^{2p}}^{2p}
\leq & \|e_\rho^{n}\|_{L^{2p}}^{2p}
+  C \delta t \left(\| T_\rho(t^{n+1}) \|_{L^{2p}}^{2p}+ \|e_\rho^{n+1}\|_{L^{2p}}^{2p}+
C ||e_\rho^{n+1}||_{L^{2p}}^{2p} \right).
\end{aligned}
\end{equation}
By summing and application of the discrete Gronwall Lemma, we obtain
\begin{equation}\label{rg5+}
\|e_\rho^{n+1}\|_{L^{2p}}^{2p}
\leq \left(\|e_\rho^{0}\|_{L_{2p}}^{2p} + C T \| T_\rho(t^{n+1}) \|_{L^{2p}}^{2p}\right)
\exp(CT).
\end{equation}
Since $e_\rho^{0}=0$ and $\|T_\rho(t^{n+1}) \|_{L^{2p}} \leq C \delta t$, we have $\|e_\rho^{n+1}\|_{L^{2p}} \leq C \delta t$.

In the parabolic-parabolic case,
setting $\eta_n=\max\limits_{k\leq n} \|e^k_\rho\|^{2p}_{L^{2p}}$ and gathering \eqref{fin} with \eqref{rg1} yield
\begin{equation}\label{fin6}
(1-C\delta t)\eta_{n+1}
\leq  \eta_{n} +  C (\delta t)^{2p}.
\end{equation}
Then the bound of $\|e_\rho^{n+1}\|_{L^{2p}}^{2p}$ follows the discrete Gronwall lemma.

To conclusion, whether the equation of $c$ is elliptic or parabolic, we have
\begin{align}\label{erhop10}
\|e_\rho^{n+1}\|_{L^{2p}}
\leq C \delta t, \ \ n=0,1,2,\ldots,N_t-1, \ \ 1<p<\infty.
\end{align}

Next, we estimate $L^p$-bound of the error $e_c^{n+1}$.
By taking inner product of the error equation \eqref{ec} with $(e_c^{n+1})^{2p-1}$ and applying the Green formula, we obtain
\begin{equation}\label{ecp1}
\begin{aligned}
&\frac{\tau}{\delta t} \| e_c^{n+1} \|_{L_{2p}}^{2p}
+\frac{2p-1}{p^2} \| \nabla (e_c^{n+1})^p \|_{L^2}^2
+\alpha \| e_c^{n+1} \|_{L_{2p}}^{2p} \\
= & \frac{\tau}{\delta t} \left( e_c^{n}, (e_c^{n+1})^{2p-1} \right)
+ \gamma \left( e_\rho^{n+1}, (e_c^{n+1})^{2p-1} \right)
+ \tau \left( T_{c}(t^{n+1}), (e_c^{n+1})^{2p-1} \right).
\end{aligned}
\end{equation}
Using H\"{o}lder inequality and Young's inequality, we have
\begin{align}
&\gamma \left( e_\rho^{n+1}, (e_c^{n+1})^{2p-1} \right)\nonumber\\
\leq &\gamma  \| e_\rho^{n+1} \|_{L_{2p}} \| e_c^{n+1} \|_{L_{2p}}^{2p-1} \label{ecp2}\\
\leq &\gamma \left( \frac{1}{2p} \|e_\rho^{n+1}\|_{L^{2p}}^{2p}
+\frac{2p-1}{2p} \|e_c^{n+1}\|_{L^{2p}}^{2p} \right), \label{ecp2add}
\end{align}
\begin{equation}\label{ecp3}
\begin{aligned}
\frac{\tau}{\delta t} \left( e_c^{n}, (e_c^{n+1})^{2p-1} \right)
\leq &\frac{\tau}{\delta t} \left( \frac{1}{2p} \|e_c^{n}\|_{L^{2p}}^{2p}
+\frac{2p-1}{2p} \|e_c^{n+1}\|_{L^{2p}}^{2p} \right),
\end{aligned}
\end{equation}
\begin{equation}\label{ecp4}
\begin{aligned}
\tau \left( T_{c}(t^{n+1}),  (e_c^{n+1})^{2p-1}\right)
\leq & \tau \left( \frac{1}{2p} \| T_c(t^{n+1}) \|_{L^{2p}}^{2p} + \frac{2p-1}{2p}
\| e_c^{n+1} \|_{L^{2p}}^{2p} \right).
\end{aligned}
\end{equation}
If $\tau=0$, then by \eqref{ecp2}
\begin{equation}\label{ecp5}
\begin{aligned}
\| e_c^{n+1} \|_{L_{2p}}
\leq  \frac{\gamma}{\alpha} \| e_\rho^{n+1} \|_{L_{2p}} \leq C \delta t, \ \  n=0,1,2,\ldots,N_t-1, \ \ 1<p<\infty.
\end{aligned}
\end{equation}
If $\tau > 0$, then substituting \eqref{ecp2add}-\eqref{ecp4} into \eqref{ecp1} yields
\begin{equation}\label{ecp5add}
\begin{aligned}
\| e_c^{n+1} \|_{L_{2p}}^{2p}
\leq &\frac{1}{2p} \|e_c^{n}\|_{L^{2p}}^{2p}
+ \frac{\gamma \delta t}{2p \tau} \|e_\rho^{n+1}\|_{L^{2p}}^{2p}
+ \frac{\delta t}{2p} \| T_c(t^{n+1}) \|_{L^{2p}}^{2p} \\
&+ \left( \frac{(2p-1)(\gamma+\tau)}{2p\tau} - \frac{\alpha}{\tau} \right) \delta t \| e_c^{n+1} \|_{L_{2p}}^{2p}.
\end{aligned}
\end{equation}
We take the sum of the above inequality \eqref{ecp5add} on $n$
and use the discrete Gronwall lemma to obtain
\begin{equation}\label{ecp7}
\begin{aligned}
\|e_c^{n+1}\|_{L^{2p}}^{2p}
\leq &\frac{1}{2p}\left(\|e_c^{0}\|_{L_{2p}}^{2p}
+ \frac{\gamma}{\tau}T \| e_\rho^{n+1} \|_{L^{2p}}^{2p}
+ T \| T_c(t^{n+1}) \|_{L^{2p}}^{2p} \right)\\
&\exp\left(\left( \frac{(2p-1)(\gamma+\tau)}{2p\tau} - \frac{\alpha}{\tau} \right) T \right).
\end{aligned}
\end{equation}
Since that $e_c^{0}=0$, $\|e_\rho^{n+1}\|_{L^{2p}}\leq C \delta t$ and $\|T_c(t^{n+1}) \|_{L^{2p}} \leq C \delta t$, we conclude
\begin{align}\label{ecp8}
\|e_c^{n+1}\|_{L^{2p}}
\leq C \delta t, \ \  n=0,1,2,\ldots,N_t-1, \ \ 1<p<\infty.
\end{align}

In summary, the proof is completed by estimates \eqref{erhop10} and \eqref{ecp8}.
\end{proof}

\section{Concluding remarks}
In this paper, we focused on numerical analysis of the semi-implicit Euler scheme \eqref{S1}-\eqref{S2} for the general KS model \eqref{2.1}-\eqref{1.4} in 2D, including the parabolic-elliptic system ($\tau=0$) and the parabolic-parabolic system ($\tau>0$).
We rigorously proved that the scheme \eqref{S1}-\eqref{S2} satisfies unconditionally all essential properties of the KS model, including the mass conservation,  positivity of the cell density, and the energy dissipation.
Furthermore, we derived the $L^p$-bounds $(1<p<\infty)$ for the  numerical solution $\rho^{n+1}$, and
carried out rigorous error analysis with optimal  error estimates in $L^p$-norm $(1<p<\infty)$.

There are several immediate directions for  future work: (i) constructing fully discrete schemes, based on the time discretization \eqref{S1}-\eqref{S2},  that can  preserve all essential properties of the KS equations at the fully discrete level; (ii) constructing second- and higher-order semi-discrete schemes that can satisfy essential properties of the KS equations; and (iii) extending the results of this paper to the three-dimensional case.

\end{document}